\UseRawInputEncoding
\documentclass[10pt]{article}
\usepackage{amsmath}

\usepackage{amsmath,amssymb,amsthm,graphicx,mathrsfs}
\usepackage[numbers,sort&compress]{natbib}
\usepackage{color}
\usepackage[colorlinks,bookmarksopen,bookmarksnumbered,citecolor=black, linkcolor=blue, urlcolor=black]{hyperref}
\numberwithin{equation}{section}
\newcommand{\beq}{\begin{equation}}
\newcommand{\enq}{\end{equation}}

\newtheorem{Theorem}{Theorem}[section]
\newtheorem{Lemma}[Theorem]{Lemma}

\newtheorem{Proposition}[Theorem]{Proposition}
\newtheorem{Remark}[Theorem]{Remark}

\newcommand{\benu}{\begin{enumerate}}
\newcommand{\beqa}{\begin{eqnarray}}
\newcommand{\beqan}{\begin{eqnarray*}}
\newcommand{\eay}{\end{array}}
\newcommand{\edm}{\end{displaymath}}
\newcommand{\eenu}{\end{enumerate}}
\newcommand{\eeq}{\end{equation}}
\newcommand{\eeqa}{\end{eqnarray}}
\newcommand{\eeqan}{\end{eqnarray*}}

\newcommand{\br}{\begin{Remark}}
\newcommand{\er}{\end{Remark}}

\newcommand{\bqa}{\begin{eqnarray}}
\newcommand{\eqa}{\end{eqnarray}}
\newcommand{\bqw}{\begin{eqnarray*}}
\newcommand{\eqw}{\end{eqnarray*}}

\newcommand{\bea}{\begin{array}{cc}}
\newcommand{\ena}{\end{array}}

\allowdisplaybreaks[4]
\begin{document}
\begin{center}
{\large \bf Local well-posedness    of solutions to 2D magnetic Prandtl model in the Prandtl-Hartmann regine}\\
\vspace{0.25in} ${\rm \rm Yuming\; Qin^{1}}\quad{\rm Xiuqing\; Wang^{2 * }}   \quad{\rm Junchen \; Liu^{2  }} $\\
\vspace{0.15in} 1. Department of Mathematics, Institute for Nonlinear Sciences,
Donghua University,\\
Shanghai 201620, P. R. China.\\
E-mail: yuming\_qin@hotmail.com\\

2. Department of Mathematics, Kunming University of Science and Technology, \\ Kunming 650500,
Yunnan, China.\\
E-mails:  daqingwang@kust.edu.cn;  liujunchen@stu.kust.edu.cn  \\

 \vspace{3mm}
\end{center}

\begin{abstract}
In this paper, we prove local existence and uniqueness for the 2D Prandtl-Hartmann regine in weighted Sobolev spaces. Our proof
is based on using uniform estimates of the regularized parabolic equation and maximal principle under the Oleinik's monotonicity assumption. Compared with Nash-Moser iteration in \cite{[5]}, the data do not require high regularity.


\end{abstract}

\section{Introduction}
The system of magnetohydrodynamics is a fundamental system to describe the fluid under the influence of electromagnetic field. The following mixed Prandtl and Hartmann boundary layer equations arise from the incompressible MHD system when the physical parameters such as Reynolds number, magnetic Reynolds number and the Hartmann number satisfy some constraints in the high Reynolds numbers limit, which were proposed in \cite{[31],[4]}.

In this paper, we consider the  Prandtl-Hartmann regine equations in the domain $\{(t,x,y)\big|t>0, x\in \mathbb{T}, y\in\mathbb{R_{+}}\}$:
\begin{equation}\left\{
\begin{array}{ll}
\partial _{t}u+u\partial _{x}u+v\partial _{y}u= \partial _{y}b+\partial _{y}^{2}u-\partial _{x}P,\\
\partial _{y}u+\partial _{y}^{2}b=0, \\
\partial _{x}u+\partial _{y}v=0,\\
\lim\limits_{y\rightarrow+\infty}u(t,x,y)=U(t,x),\quad  \lim\limits_{y\rightarrow+\infty}b(t,x,y)=B(t,x).
\end{array}
 \label{1.1}         \right.\end{equation}
Here
$(u,v)=(u(t,x,y),v(t,x,y))$ denotes the velocity field, $b(t,x,y)$ is the corresponding tangential magnetic component.  The given scalar pressure $P:=P(x,t)$ and the outer flow $U$ satisfy
the well-known Bernoulli's law:
\begin{eqnarray}
\partial_{t}U+U\partial_{x}U+\partial_{x}P=0 .
\end{eqnarray}
The initial data and no-slip boundary condition are imposed by
\begin{equation}\left\{
\begin{array}{ll}
u(0,x,y)=u_{0}(x,y),\\
u(t,x,0)=0,\ \ v(t,x,0)=0.
\end{array}
 \label{1.2}         \right.\end{equation}

Noticing $(\ref{1.1})_{4}$, and integrating $(\ref{1.1})_{2}$ in $y$ over $[y, +\infty)$, we obtain
\begin{eqnarray}
\partial_{y}b=U-u . \label{1.3}
\end{eqnarray}
Inserting $(\ref{1.3})$ into $(\ref{1.1})$, then we arrive at the following equations
\begin{equation}\left\{
\begin{array}{ll}
\partial _{t}u+u\partial _{x}u+v\partial _{y}u=(U-u) +\partial _{y}^{2}u-\partial _{x}P,\\
\partial _{x}u+\partial _{y}v=0.
\end{array}
 \label{1.4}         \right.\end{equation}
Let the vorticity $w=\partial_{y}u$, then equations $(\ref{1.4})$ reduce to the following vorticity system
\begin{equation}\left\{
\begin{array}{ll}
\partial _{t}w+u\partial _{x}w+v\partial _{y}w=-w+\partial _{y}^{2}w,\\
w(0,x,y)=\partial _{y} u_{0},\\
\partial _{y}w|_{y=0}=\partial _{x}P-U.
\end{array}
 \label{1.5}         \right.\end{equation}

In this paper, we consider system $(\ref{1.1})$ under the Oleinik's monotonicity assumption $w=\partial_{y}u>0$. 
Now let us briefly review the background and corresponding results about the boundary layer. Actually, Prandtl proposed the basic rules for describing a phenomenon at the Heidelberg international mathematical conference in 1904. He pointed out that there are two regions in the flow of a solid: a thin layer near an object, viscous friction plays an important role; outside this thin layer, friction is negligible.  Prandtl called this thin layer as the boundary layer, which can be described by the so-called Prandtl equations. Later, the researches on the boundary layer were appeared, such as \cite{[29],[30]}. But until 1999, Oleinik and Samokhin \cite{[26]} gave a systematic study of the Prandtl equations from a mathematical point of view, which is the fundamental research on boundary layer system.  The Prandtl system is obtained as  a simplification of the Navier-Stokes system and describes the motion of a fluid with small viscosity about a solid body in a thin layer which is formed near its surface owing to the adhesion of the viscous fluid to the solid surface.

The first well-known result was developed by Oleinik and Samokhin in \cite{[26]}, where under the monotonicity condition on tangential velocity with respect to the normal variable to the boundary, the local (in time) well-posedness of Prandtl equations was obtained by using the Crocco transformation and von Mises transformation.  Since then, Crocco transformation and von Mises transformation have been used in boundary layer problems. For example, Xin and Zhang \cite{[27]} established a global existence of weak solutions to the two-dimensional Prandtl system for the pressure is favourable, which generalized the local well-posedness results of Oleinik \cite{[26]}. Gong, Guo and Wang \cite{[7]} investigated the local spatial existence of solution for the compressible Navier-Stokes boundary layer equations.
Coordinate transformation was also used to MHD boundary layer in \cite{[20]}. The authors in \cite{[5],[17],[18],[8],[10],[9]} studied the well-posedness of the Prandtl equations around a shear flow,  i.e., let $u(t,x,y)=u^{s}(t,y)+\widetilde{u}(t,x,y)$, and shear flow $u^{s}(t,y)$ is a solution of   heat equations
\begin{equation*}\left\{
\begin{array}{ll}
\partial_{t}u^{s}-\partial_{y}^{2}u^{s}=0,\\
u^{s}|_{y=0}=0, \ \ \lim\limits_{y\rightarrow+\infty}u^{s}(t,y)=1,\\
u^{s}|_{t=0}=u^{s}_{0}(y).
\end{array}
 \label{ }         \right.\end{equation*}
Xie and Yang in \cite{[19]} also obtained the local existence of solutions to the MHD boundary layer system as a general shear flow.
G$\acute{e}$ratd-Varet in \cite{[21],[34]} showed the local in time well-posedness of the 2D Prandtl equations for data with Gevery regularity.
Masmoudi and Wong \cite{[24],[2]} obtained the local existence and uniqueness for the two-dimensional Prandtl system in weighted Sobolev spaces under Oleinik's monotonicity assumption by using pure energy method which based on a cancellation property. Thereafter, Xie and Yang \cite{[13],[14]}, Gao, Huang and Yao \cite{[15]} studied the local well-posedness of MHD boundary layer problems by the energy method. Recently, the first author and Dong \cite{[39]} established the local well-posedness of solutions in $H^s$ with $1\leq s\leq 4\; U\equiv constant$ and without monotonicity condition and a lower bound. While in this paper, we prove the local well-posedness of solutions in $H^s$ to 2D magnetic Prandtl equations in Prandtl-Hartmann regine  with $s\geq 4$, a function $U(t,x)$ and different homogeneous equations from \cite{[39]}.

There are also many references about the ill-posedness of Prandtl equations, e.g., see \cite{[31],[32],[35],[23]}. For more  references, we refer to \cite{[28],[36],[11],[16],[6],[25]}.

From above statement, we know that it is hard to prove the local existence of solutions by the energy method.
Our purpose of this paper is to prove the local existence of solutions of MHD boundary layer in the Prandtl-Hartmann regine, whose proof  is based on  a nonlinear energy estimate.

We introduce the weighted sobolev space and define the space $H^{s,\gamma}_{\sigma,\delta}$ by
$$H^{s,\gamma}_{\sigma,\delta}:=\left\{w:\mathbb{T}\times\mathbb{R_{+}}\rightarrow\mathbb{R}:
\|w\|_{H^{s,\gamma}}<+\infty, (1+y)^{\sigma}|w|\geq\delta,
\sum \limits_{|\alpha|\leq 2}  |(1+y)^{\sigma+\alpha_{2}}D^{\alpha}w | ^{2}\leq\frac{1}{\delta^{2}} \right\},$$
where $D^{\alpha}:=\partial_{x}^{\alpha_{1}}\partial_{y}^{\alpha_{2}}$, $\alpha_{1}+\alpha_{2}=s$,  $s\geq 4$, $\gamma\geq 1$, $\sigma>\gamma+\frac{1}{2}$ and $\delta\in (0,1)$. We define the norm as
 $$\|w\|_{H^{s,\gamma}}^{2}:=\sum \limits_{|\alpha|\leq s} \|(1+y)^{\gamma+\alpha_{2}}D^{\alpha}w\|_{L^{2}}^{2}$$
and
$$\|w\|_{H^{s,\gamma}_{g}}^{2}:=\|(1+y)^{\gamma}g_{s}\|^{2}_{L^{2}}+
\sum\limits_{\substack{ |\alpha|\leq s\\ \alpha_{1}\leq s-1}}  \|(1+y)^{\gamma+\alpha_{2}}D^{\alpha}w\|_{L^{2}}^{2},$$
here
$$g_{s}:=\partial_{x}^{s}w-\frac{\partial_{y}w}{w}\partial_{x}^{s}(u-U).$$
In this paper, for convenience, we simply write
$$\iint\cdot:=\int_{\mathbb{T}}\int_{\mathbb{R_{+}}}\cdot dxdy.$$

The major difficulty for the study of the  Prandtl-Hartmann regine equations (\ref{1.1}) is  the vertical velocity $v=-\partial_y^{-1} \partial_x u$, which leads to a regularity loss of $x$-derivative. Thus, we introduce a weighted norm $\|w\|_{H^{s,\gamma}_{g}}^{2}$ (it was developed in work \cite{[2]}) for the vorticity $w$ which avoids the loss of $x$-derivative by the nonlinear cancellation.  Moreover, the lack of higher-order boundary conditions at $y=0$ prevents us from applying the integration by parts in the $y$-variable. An estimate of boundary data for the higher-order boundary conditions can help us overcome this technical difficulty, see Lemma \ref{y2.1}. Note that the equations studied in this paper are different from those in \cite{[2]}. It worth pointing out here that an innovation point in this paper is that we do not use the classical Crocco transformation and von Mises transformation introduced in \cite{[26]}, but only under the assumption on monotonicity. Firstly, to prove the local existence of solutions, we will construct an approximate scheme and study the parabolic regularized Prandtl equation (\ref{2.0001}), which preserves the nonlinear structure of the  equation (\ref{1.1}), as well as the nonlinear cancellation properties.

The article is structured as follows. Section 1 is the introduction, which contains many inequalities used in this paper. In Section 2, we give the uniform estimates with the weighted norm, the difficulty of this part is the estimation on the boundary at $y=0$, and we give the regular pattern at $y=0$ eventually, see Lemmas \ref{y2.1}-\ref{y2.2}.
In Section 3, we prove local existence and uniqueness of solutions to the Prandtl system.

Last, we state our main result as follows.
\begin{Theorem}\label{t1.1}
Given any even integer  $s \geq 4$, and real numbers $\gamma, \sigma, \delta$ satisfying $\gamma\geq 1$, $\sigma>\gamma+\frac{1}{2}$ and $\delta\in (0,1)$.   Assume the following conditions hold,

(i) suppose that the initial data $u_0-U(0,x) \in H^{s, \gamma-1}$ and $\partial_y u_0 \in H^{s,\gamma}_{\sigma, 2\delta}$ satisfy the compatibility conditions $u_0|_{y=0}$ and $\lim \limits_{y \rightarrow +\infty} u=U|_{t=0}$. In addition, when $s = 4$,   assume  that $\delta \geq 0$ is chosen small enough such that $\|\omega_0\|_{H^{s,\gamma} } \leq C \delta^{-1}$ with a generic constant $C$.

(ii) the outer flow $U$  satisfies
 \begin{eqnarray}
\sup \limits_{t} \sum\limits_{l=0}^{\frac{s}{2}+1}\|\partial_t^l U\|_{H^{s-2l+2}(\mathbb{T})} < + \infty.
\label{1.8}
\end{eqnarray}

Then there exists a time $T := T(s, \gamma, \sigma, \delta,  \|w_0\|_{H^{s,\gamma} },U)$ such that the initial-boundary value problem
(\ref{1.1})-(\ref{1.2}) has a unique local classical solution $(u, v, b)$ satisfying
$$u-U \in L^{\infty}([0,T];H^{s,\gamma-1}) \cap C([0,T]; H^{s}-w)$$
and
$$ \partial_y u \in L^{\infty}([0,T];H^{s,\gamma}_{\sigma, \delta}) \cap C([0,T]; H^{s}-w),$$
where $H^s-w$ is the space $H^s$ endowed with its weak topology.
\end{Theorem}

\subsection{Preliminaries}
In this subsection, we will give some lemmas to our main result.
\begin{Lemma}(\cite{[2]})\label{p4.1}
Let $s\geq 4$ be an integer, $\gamma\geq 1$, $\sigma>\gamma+\frac{1}{2}$ and $\delta\in (0,1)$. Then for any $w\in H^{s,\gamma}_{\sigma,\delta}(\mathbb{T}\times\mathbb{R}_{+})$,
we have the following inequality
 \begin{eqnarray}
C_{\delta }\|w\|_{H^{s,\gamma}_{g}}\leq  \|w\|_{H^{s,\gamma} }+ \|u-U\|_{H^{s,\gamma-1} } \leq C_{s, \gamma,\sigma,\delta  }\left(\|w\|_{H^{s,\gamma}_{g}}+\|\partial_{x}^{s}U\|_{{L^2}(\mathbb{T})}\right)
  ,
\label{4.1}
\end{eqnarray}
where $C_{s, \gamma,\sigma,\delta  }>0$ is a constant and only depends on $s, \gamma,\sigma,\delta  $.
\end{Lemma}

\begin{Lemma}\label{y4.1}
Let $f:\mathbb{T} \times\mathbb{R}_{+}\rightarrow \mathbb{R}$, \\
(i)  if $\lambda>-\frac{1}{2}$  and $ \lim \limits_{y\rightarrow +\infty }f(x,y ) =0$,  then
 \begin{eqnarray}
 \|(1+y)^{\lambda}f\|_{L^{2}(\mathbb{T} \times\mathbb{R}_{+})} \leq 
\frac{2}{2\lambda+1} \|(1+y)^{\lambda+1}\partial_{y}f\|_{L^{2}(\mathbb{T} \times\mathbb{R}_{+})};
\label{4.10}
\end{eqnarray}
(ii)    if $\lambda<-\frac{1}{2}$  and $ f(x,y )|_{y=0}=0$,  then
 \begin{eqnarray}
 \|(1+y)^{\lambda}f\|_{L^{2}(\mathbb{T} \times\mathbb{R}_{+})} \leq 
 -\frac{2}{2\lambda+1} \|(1+y)^{\lambda+1}\partial_{y}f\|_{L^{2}(\mathbb{T} \times\mathbb{R}_{+})}.
\label{4.11}
\end{eqnarray}
\end{Lemma}
\begin{proof}
The proof is elementary. Note that the term
\begin{eqnarray*}
&&\|(1+y)^{\lambda}f\|_{L^{2}(\mathbb{T} \times\mathbb{R}_{+})} ^{2}
=\frac{1}{2\lambda+1}\int_{\mathbb{T}}\int_{\mathbb{R}_{+}}f^{2}( x,y)d(1+y)^{2\lambda+1} dx  \nonumber \\
&&=\frac{1}{2\lambda+1}\int_{\mathbb{T}}(1+y)^{2\lambda+1}f^{2}( x,y) \big|_{y=0}^{y=+\infty}dx
-\frac{2}{2\lambda+1} \int_{\mathbb{T}}\int_{\mathbb{R}_{+}} (1+y)^{2\lambda+1}f ( x,y) \partial_{y}f(x,y)dxdy \nonumber \\
&& \leq \left|-\frac{2}{2\lambda+1}\right| \|(1+y)^{\lambda }  f\|_{L^{2}(\mathbb{T} \times\mathbb{R}_{+})}  \|(1+y)^{\lambda+1}\partial_{y}f\|_{L^{2}(\mathbb{T} \times\mathbb{R}_{+})},
 \end{eqnarray*}
which gives $(\ref{4.10})$-$(\ref{4.11})$.
\end{proof}

\begin{Lemma}(\cite{[2]})\label{y4.2}
Let $s\geq 4$ be an integer, $\gamma\geq 1$, $\sigma>\gamma+\frac{1}{2}$ and $\delta\in (0,1)$. For any $w\in H^{s,\gamma}_{\sigma,\delta}(\mathbb{T}\times\mathbb{R}_{+})$, then  for $k=0,1,2,\cdots, s$,
we have
 \begin{eqnarray}
 &&\|(1+y)^{\gamma} g_{k}\|_{L^{2}}\leq\|(1+y)^{\gamma}\partial _{x}^{k}w\|_{L^{2}}+\delta^{-2}\|(1+y)^{\gamma-1}\partial _{x}^{k}(u-U)\|_{L^{2}}  ,
 \label{4.5}
\end{eqnarray}
and
 \begin{eqnarray}
 &&\|(1+y)^{\gamma}\partial _{x}^{k}w\|_{L^{2}}+\|(1+y)^{\gamma-1}\partial _{x}^{k}(u-U)\|_{L^{2}}\leq C_{\gamma,\sigma,\delta} \left(\|\partial_{x}^{k}U\|_{{L^2}(\mathbb{T})}+\|(1+y)^{\gamma }g_{k}\|_{L^{2}}\right),\label{4.4}
\end{eqnarray}
where $C_{s, \gamma,\sigma,\delta  }>0$ is a constant and only depends on $s, \gamma,\sigma,\delta  $.\\
%
\end{Lemma}

\begin{Lemma}(\cite{[2]})\label{y4.3}
Let $f:\mathbb{T} \times\mathbb{R}_{+}\rightarrow \mathbb{R}$.  Then there exists a constant $C>0$ such that
 \begin{eqnarray}
 \|f(x,y )\|_{L^{\infty}(\mathbb{T} \times \mathbb{R}_{+})}    \leq C
\left(\|f(x,y)\|_{L^{2}(\mathbb{T}\times \mathbb{R}_{+} )}   +\|\partial_{x}f(x,y)\|_{L^{2}(\mathbb{T}\times \mathbb{R}_{+} )}  +\|\partial_{y }^{2}f(x,y)\|_{L^{2}(\mathbb{T}\times \mathbb{R}_{+} )} \right).
\label{4.12}
\end{eqnarray}

\end{Lemma}

\begin{Lemma}(\cite{[2]})\label{y4.4}
Let $s\geq 4$ is an integer, $\gamma\geq 1$, $\sigma>\gamma+\frac{1}{2}$ and $\delta\in (0,1)$. Then for any $w\in H^{s,\gamma}_{\sigma,\delta}(\mathbb{T}\times\mathbb{R}_{+})$,
we have the following inequalities: \\
$(i)$ for $k=0, 1, 2,  \cdots, s-1 $,
\begin{eqnarray}
 \| \frac{\partial_{x}^{k} v +y\partial_x^{k+1}}{1+y}  \|_{L^{2}} \leq C_{s, \gamma,\sigma,\delta  } (\|w\|_{H^{s,\gamma}_{g}}+\|\partial_{x}^{s}U\|_{{L^2}(\mathbb{T})});
\label{4.13}
\end{eqnarray}
(ii) for $k=0, 1, 2,  \cdots, s $,
\begin{eqnarray}
 \|  (1+y)^{\gamma-1}\partial_x^{k}(u-U) \|_{L^{2}} \leq C_{s, \gamma,\sigma,\delta  } (\|w\|_{H^{s,\gamma}_{g}}+\|\partial_{x}^{s}U\|_{{L^2}(\mathbb{T})});
\label{4.14}
\end{eqnarray}
(iii) for $k=0, 1, 2,  \cdots, s-2 $,
\begin{eqnarray}
 \|  \frac{\partial_{x}^{k} v }{1+y}  \|_{L^{\infty}} \leq C_{s, \gamma,\sigma,\delta  } (\|w\|_{H^{s,\gamma}_{g}}+\|\partial_{x}^{s}U\|_{{L^2}(\mathbb{T})});
\label{4.15}
\end{eqnarray}
(iv) for $k=0, 1, 2,  \cdots, s-1 $,
\begin{eqnarray}
 \|  \partial_{x}^{k}u  \|_{L^{\infty}} \leq C_{s, \gamma,\sigma,\delta  }  (\|w\|_{H^{s,\gamma}_{g}}+\|\partial_{x}^{s}U\|_{{L^2}(\mathbb{T})});
\label{4.16}
\end{eqnarray}
(v) for $k=0, 1, 2,  \cdots, s-2$,
\begin{eqnarray}
 \|  (1+y)^{\gamma+k_{\alpha_{2}} }D^{k  }w  \|_{L^{\infty}} \leq C_{s, \gamma,\sigma,\delta  } \|w\|_{H^{s,\gamma}_{g}};
\label{4.17}
\end{eqnarray}
(vi) for all $|\alpha| \leq s$,
\begin{eqnarray}
\|(1+y)^{\gamma+\alpha_2}D^{\alpha} w\|_{L^2} \leq
\left\{
    \begin{array}{ll}
        C_{s, \gamma,\sigma,\delta  }  (\|w\|_{H^{s,\gamma}_{g}}+\|\partial_{x}^{s}U\|_{{L^2}(\mathbb{T})}) & if~ \alpha=(s,0), \\
       \|w\|_{H^{s,\gamma}_{g}} & if ~ \alpha \neq (s,0);
    \end{array}
\right.\label{4.52}
\end{eqnarray}
(vii) for all $k=0, 1, 2,  \cdots, s$,
\begin{eqnarray}
\|(1+y)^{\gamma}g_k\|_{L^2(\mathbb{T})} \leq
\left\{
    \begin{array}{ll}
        C_{s, \gamma,\sigma,\delta  }  (\|w\|_{H^{s,\gamma}_{g}}+\|\partial_{x}^{s}U\|_{{L^2}(\mathbb{T})}) & if~ k=0, 1, 2,  \cdots, s-1, \\
       \|w\|_{H^{s,\gamma}_{g}} & if ~ k=s,
    \end{array}
\right.\label{4.51}
\end{eqnarray}
where $C_{s, \gamma,\sigma,\delta  }$ is a constant and only depends on $s, \gamma,\sigma,\delta  $.
\end{Lemma}
\subsection{Preliminaries}
In this subsection, we will give some lemmas to our main result.
\begin{Lemma}(\cite{[2]})\label{p4.1}
Let $s\geq 4$ be an integer, $\gamma\geq 1$, $\sigma>\gamma+\frac{1}{2}$ and $\delta\in (0,1)$. Then for any $w\in H^{s,\gamma}_{\sigma,\delta}(\mathbb{T}\times\mathbb{R}_{+})$,
we have the following inequality
 \begin{eqnarray}
C_{\delta }\|w\|_{H^{s,\gamma}_{g}}\leq  \|w\|_{H^{s,\gamma} }+ \|u-U\|_{H^{s,\gamma-1} } \leq C_{s, \gamma,\sigma,\delta  }\left(\|w\|_{H^{s,\gamma}_{g}}+\|\partial_{x}^{s}U\|_{{L^2}(\mathbb{T})}\right)
  ,
\label{4.1}
\end{eqnarray}
where $C_{s, \gamma,\sigma,\delta  }>0$ is a constant and only depends on $s, \gamma,\sigma,\delta  $.
\end{Lemma}

\begin{Lemma}\label{y4.1}
Let $f:\mathbb{T} \times\mathbb{R}_{+}\rightarrow \mathbb{R}$, \\
(i)  if $\lambda>-\frac{1}{2}$  and $ \lim \limits_{y\rightarrow +\infty }f(x,y ) =0$,  then
 \begin{eqnarray}
 \|(1+y)^{\lambda}f\|_{L^{2}(\mathbb{T} \times\mathbb{R}_{+})} \leq 
\frac{2}{2\lambda+1} \|(1+y)^{\lambda+1}\partial_{y}f\|_{L^{2}(\mathbb{T} \times\mathbb{R}_{+})};
\label{4.10}
\end{eqnarray}
(ii)    if $\lambda<-\frac{1}{2}$  and $ f(x,y )|_{y=0}=0$,  then
 \begin{eqnarray}
 \|(1+y)^{\lambda}f\|_{L^{2}(\mathbb{T} \times\mathbb{R}_{+})} \leq 
 -\frac{2}{2\lambda+1} \|(1+y)^{\lambda+1}\partial_{y}f\|_{L^{2}(\mathbb{T} \times\mathbb{R}_{+})}.
\label{4.11}
\end{eqnarray}
\end{Lemma}
\begin{proof}
The proof is elementary. Note that the term
\begin{eqnarray*}
&&\|(1+y)^{\lambda}f\|_{L^{2}(\mathbb{T} \times\mathbb{R}_{+})} ^{2}
=\frac{1}{2\lambda+1}\int_{\mathbb{T}}\int_{\mathbb{R}_{+}}f^{2}( x,y)d(1+y)^{2\lambda+1} dx  \nonumber \\
&&=\frac{1}{2\lambda+1}\int_{\mathbb{T}}(1+y)^{2\lambda+1}f^{2}( x,y) \big|_{y=0}^{y=+\infty}dx
-\frac{2}{2\lambda+1} \int_{\mathbb{T}}\int_{\mathbb{R}_{+}} (1+y)^{2\lambda+1}f ( x,y) \partial_{y}f(x,y)dxdy \nonumber \\
&& \leq \left|-\frac{2}{2\lambda+1}\right| \|(1+y)^{\lambda }  f\|_{L^{2}(\mathbb{T} \times\mathbb{R}_{+})}  \|(1+y)^{\lambda+1}\partial_{y}f\|_{L^{2}(\mathbb{T} \times\mathbb{R}_{+})},
 \end{eqnarray*}
which gives $(\ref{4.10})$-$(\ref{4.11})$.
\end{proof}

\begin{Lemma}(\cite{[2]})\label{y4.2}
Let $s\geq 4$ be an integer, $\gamma\geq 1$, $\sigma>\gamma+\frac{1}{2}$ and $\delta\in (0,1)$. For any $w\in H^{s,\gamma}_{\sigma,\delta}(\mathbb{T}\times\mathbb{R}_{+})$, then  for $k=0,1,2,\cdots, s$,
we have
 \begin{eqnarray}
 &&\|(1+y)^{\gamma} g_{k}\|_{L^{2}}\leq\|(1+y)^{\gamma}\partial _{x}^{k}w\|_{L^{2}}+\delta^{-2}\|(1+y)^{\gamma-1}\partial _{x}^{k}(u-U)\|_{L^{2}}  ,
 \label{4.5}
\end{eqnarray}
and
 \begin{eqnarray}
 &&\|(1+y)^{\gamma}\partial _{x}^{k}w\|_{L^{2}}+\|(1+y)^{\gamma-1}\partial _{x}^{k}(u-U)\|_{L^{2}}\leq C_{\gamma,\sigma,\delta} \left(\|\partial_{x}^{k}U\|_{{L^2}(\mathbb{T})}+\|(1+y)^{\gamma }g_{k}\|_{L^{2}}\right),\label{4.4}
\end{eqnarray}
where $C_{s, \gamma,\sigma,\delta  }>0$ is a constant and only depends on $s, \gamma,\sigma,\delta  $.\\
%
\end{Lemma}

\begin{Lemma}(\cite{[2]})\label{y4.3}
Let $f:\mathbb{T} \times\mathbb{R}_{+}\rightarrow \mathbb{R}$.  Then there exists a constant $C>0$ such that
 \begin{eqnarray}
 \|f(x,y )\|_{L^{\infty}(\mathbb{T} \times \mathbb{R}_{+})}    \leq C
\left(\|f(x,y)\|_{L^{2}(\mathbb{T}\times \mathbb{R}_{+} )}   +\|\partial_{x}f(x,y)\|_{L^{2}(\mathbb{T}\times \mathbb{R}_{+} )}  +\|\partial_{y }^{2}f(x,y)\|_{L^{2}(\mathbb{T}\times \mathbb{R}_{+} )} \right).
\label{4.12}
\end{eqnarray}

\end{Lemma}

\begin{Lemma}(\cite{[2]})\label{y4.4}
Let $s\geq 4$ is an integer, $\gamma\geq 1$, $\sigma>\gamma+\frac{1}{2}$ and $\delta\in (0,1)$. Then for any $w\in H^{s,\gamma}_{\sigma,\delta}(\mathbb{T}\times\mathbb{R}_{+})$,
we have the following inequalities: \\
$(i)$ for $k=0, 1, 2,  \cdots, s-1 $,
\begin{eqnarray}
 \| \frac{\partial_{x}^{k} v +y\partial_x^{k+1}}{1+y}  \|_{L^{2}} \leq C_{s, \gamma,\sigma,\delta  } (\|w\|_{H^{s,\gamma}_{g}}+\|\partial_{x}^{s}U\|_{{L^2}(\mathbb{T})});
\label{4.13}
\end{eqnarray}
(ii) for $k=0, 1, 2,  \cdots, s $,
\begin{eqnarray}
 \|  (1+y)^{\gamma-1}\partial_x^{k}(u-U) \|_{L^{2}} \leq C_{s, \gamma,\sigma,\delta  } (\|w\|_{H^{s,\gamma}_{g}}+\|\partial_{x}^{s}U\|_{{L^2}(\mathbb{T})});
\label{4.14}
\end{eqnarray}
(iii) for $k=0, 1, 2,  \cdots, s-2 $,
\begin{eqnarray}
 \|  \frac{\partial_{x}^{k} v }{1+y}  \|_{L^{\infty}} \leq C_{s, \gamma,\sigma,\delta  } (\|w\|_{H^{s,\gamma}_{g}}+\|\partial_{x}^{s}U\|_{{L^2}(\mathbb{T})});
\label{4.15}
\end{eqnarray}
(iv) for $k=0, 1, 2,  \cdots, s-1 $,
\begin{eqnarray}
 \|  \partial_{x}^{k}u  \|_{L^{\infty}} \leq C_{s, \gamma,\sigma,\delta  }  (\|w\|_{H^{s,\gamma}_{g}}+\|\partial_{x}^{s}U\|_{{L^2}(\mathbb{T})});
\label{4.16}
\end{eqnarray}
(v) for $k=0, 1, 2,  \cdots, s-2$,
\begin{eqnarray}
 \|  (1+y)^{\gamma+k_{\alpha_{2}} }D^{k  }w  \|_{L^{\infty}} \leq C_{s, \gamma,\sigma,\delta  } \|w\|_{H^{s,\gamma}_{g}};
\label{4.17}
\end{eqnarray}
(vi) for all $|\alpha| \leq s$,
\begin{eqnarray}
\|(1+y)^{\gamma+\alpha_2}D^{\alpha} w\|_{L^2} \leq
\left\{
    \begin{array}{ll}
        C_{s, \gamma,\sigma,\delta  }  (\|w\|_{H^{s,\gamma}_{g}}+\|\partial_{x}^{s}U\|_{{L^2}(\mathbb{T})}) & if~ \alpha=(s,0), \\
       \|w\|_{H^{s,\gamma}_{g}} & if ~ \alpha \neq (s,0);
    \end{array}
\right.\label{4.52}
\end{eqnarray}
(vii) for all $k=0, 1, 2,  \cdots, s$,
\begin{eqnarray}
\|(1+y)^{\gamma}g_k\|_{L^2(\mathbb{T})} \leq
\left\{
    \begin{array}{ll}
        C_{s, \gamma,\sigma,\delta  }  (\|w\|_{H^{s,\gamma}_{g}}+\|\partial_{x}^{s}U\|_{{L^2}(\mathbb{T})}) & if~ k=0, 1, 2,  \cdots, s-1, \\
       \|w\|_{H^{s,\gamma}_{g}} & if ~ k=s,
    \end{array}
\right.\label{4.51}
\end{eqnarray}
where $C_{s, \gamma,\sigma,\delta  }$ is a constant and only depends on $s, \gamma,\sigma,\delta  $.
\end{Lemma}

\section{Uniform estimates on the regularized system}
In this section, we will estimate the norm $\|w\|_{H^{s,\gamma}_{g}}^{2}$  by the energy method. We consider the regularized equations to the problem $(\ref{1.4})$ for any $\epsilon$,
\begin{equation}\left\{
\begin{array}{ll}
\partial _{t}u^\epsilon+u^\epsilon\partial _{x}u^\epsilon+v^\epsilon\partial _{y}u^\epsilon=(U-u^\epsilon)+\epsilon^2\partial _{x}^{2}u^\epsilon +\partial _{y}^{2}u^\epsilon-\partial _{x}P^\epsilon,\\
\partial _{x}u^\epsilon+\partial _{y}v^\epsilon=0,\\
u^{\epsilon}|_{t=0}=u_0,\\
u^\epsilon|_{y=0}=v^\epsilon |_{y=0}=0,\\
\lim\limits_{y\rightarrow+\infty}u^\epsilon(t,x,y)=U(t,x),
\label{2.0001}
\end{array}
         \right.\end{equation}
where $P^\epsilon$ and $U$ satisfy a regularized Bernoulli's law:
\begin{eqnarray}
\partial_{t}U+U\partial_{x}U-\epsilon^2\partial^2_xU+\partial_{x}P^\epsilon=0 .
\end{eqnarray}
Then, the regularized vorticity $w^\epsilon=\partial_y u^\epsilon$ satisfies the following regularized vorticity system for any $\epsilon>0$,
\begin{equation}\left\{
\begin{array}{ll}
\partial _{t}w^\epsilon+u^\epsilon\partial _{x}w^\epsilon+v^\epsilon\partial _{y}w^\epsilon=-w^\epsilon+\epsilon^2\partial _{x}^{2}w^\epsilon +\partial _{y}^{2}w^\epsilon,\\
{w^\epsilon}|_{t=0}=w_0=\partial _{y}u_0,\\
\partial_y{w^\epsilon}|_{y=0}=\partial_x P^\epsilon -U,
\label{2.0002}
\end{array}
         \right.\end{equation}
where the velocity field $(u^\epsilon,v^\epsilon)$ is given
\begin{eqnarray}
u^\epsilon(t,x,y)=U-\int^{+\infty}_y w(t,x,\widetilde{y})d \widetilde{y}, \quad v^\epsilon(t,x,y)=-\int^{y}_0 \partial_x u(t,x,\widetilde{y})d \widetilde{y}.
\end{eqnarray}
From now on, we drop the superscript $\epsilon$ for simplicity of notations.
\subsection{Estimates on $D^{\alpha}w$ }\label{su.1}
In this subsection, we will estimate the norm with weight $(1+y)^{ \gamma+ \alpha_{2}}$ on $D^{\alpha}w$, $\alpha=(\alpha_{1},\alpha_{2})$, $|\alpha|\leq s$, and $\alpha_{1}$ should be restricted as $\alpha_{1}\leq s-1$. Otherwise, it is impossible to directly estimate the norm with $\partial _{x}^{s}w$.
\begin{Lemma}(Reduction of boundary data)\label{y2.1}
If $w$ solves $(\ref{2.0001})$ and $(\ref{2.0002})$, then on the boundary at $y=0$, we have,
\begin{equation}\left\{
\begin{array}{ll}
\partial _{y}w|_{y=0}=\partial _{x}P-U(t,x),\\
\partial _{y}^{3}w|_{y=0}= (\partial _{t}-\epsilon^2\partial_x^2)(\partial _{x}P-U(t,x))+(\partial _{x}P-U(t,x))+w\partial _{x}w|_{y=0}.
\end{array}
         \right.\end{equation}
For any $2 \leq k \leq \frac{s}{2}$, there are some constants $C_k$, $C_{\wedge_{\alpha},k,l,\rho^1,\rho^2,...,\rho^j}$,   not depending on $\epsilon$ or $(u,v,w)$, such that
\begin{eqnarray}
\begin{aligned}
\partial_y^{2k+1}w|_{y=0}&=C_k\sum\limits_{s=0}^{k}(\partial_t-\epsilon^2\partial^2_x)^s(\partial_x P -U)\\
&\quad+\sum\limits_{l=0}^{k-1}\epsilon^{2l}\sum\limits_{j=2}^{\max\{2,k-l\}}\sum\limits_{\rho \in A^j_{k,l}}C_{\wedge_{\alpha},k,l,\rho^1,\rho^2,...,\rho^j}\prod^{j}_{i=1}D^{\rho^i}w \big{|}_{y=0}
\label{2.005}
\end{aligned}
\end{eqnarray}
where $A^j_{k,l}:=\{\rho:=(\rho^1,\rho^2,...,\rho^j)\in \mathbb{N}^{2j};3\sum\limits_{i=1}^{j}\sum\limits_{i=1}^{j}\rho^i_2=2k+4l+1,\sum\limits_{i=1}^{j}\rho^i_1 \leq k+2l-1,\sum\limits_{i=1}^{j}\rho_2^j\leq 2k-2l-2, ~and~ |\rho^i| \leq 2k-l-1 ~for~ all~ i=1,2,...,j\}$.
\end{Lemma}
\begin{proof}
According to equation $(\ref{2.0001})$ and the boundary condition $(\ref{1.2})_{2}$, we get
$$\partial _{y}w|_{y=0}=\partial _{x}P-U(t,x)=K,$$
which implies
\begin{eqnarray}
\begin{aligned}
\partial _{y}^{3}w \big{|}_{y=0}&= \left\{\partial _{t}\partial _{y}w+\partial _{y}(u\partial _{x}w)+\partial _{y}(v\partial _{y}w)+\partial _{y}w-\epsilon^2\partial_y\partial_x^2w\right\} \big{|}_{y=0}\\
&=(\partial _{t}-\epsilon^2\partial_x^2)K+K+w\partial _{x}w \big{|}_{y=0},
\end{aligned}
\end{eqnarray}
and
\begin{eqnarray}
\begin{aligned}
\partial _{y}^{2n+1}w \big{|}_{y=0}&= \left\{\left(\partial _{t}-\epsilon^2\partial_x^2w+1\right)\partial _{y}^{2n-1}w+\partial _{y}^{2n-1}(u\partial _{x}w+v\partial _{y}w)\right\} \big{|}_{y=0}.
 \label{2.15}
\end{aligned}
\end{eqnarray}
Hence, for $n=2$,
\begin{eqnarray}
\partial _{y}^{5}w \big{|}_{y=0}
&=& (\partial _{t}-\epsilon^2\partial_x^2)\partial_y^{3}w|_{y=0}+\partial _{y}^{3}w|_{y=0}+(2\partial_y w\partial_x\partial_yw+3w\partial_x\partial_y^2w-2\partial_xw\partial_y^2w) \big{|}_{y=0}\nonumber\\
&=&(\partial _{t}-\epsilon^2\partial_x^2)^2K+(\partial _{t}-\epsilon^2\partial_x^2)K+(\partial _{t}-\epsilon^2\partial_x^2)(w\partial_x w)|_{y=0}+(\partial _{t}-\epsilon^2\partial_x^2)K\nonumber\\
&\quad&+K+w\partial _{x}w|_{y=0}+(2\partial_y w\partial_x\partial_yw+3w\partial_x\partial_y^2w-2\partial_xw\partial_y^2w) \big{|}_{y=0},
 \label{2.16}
\end{eqnarray}
here
\begin{equation}\left\{
\begin{array}{ll}
\partial _{t} w|_{y=0}= \partial _{y}^{2}w-u\partial _{x}w-v\partial _{y}w-w-\epsilon^2\partial_x^2 w|_{y=0}=\partial _{y}^{2}w-w+\epsilon^2\partial_x^2 w\big{|}_{y=0}, \\
\partial _{x}\partial _{t} w|_{y=0} =\partial _{x}\partial _{y}^{2}w-\partial _{x}w+\epsilon^2\partial_x^3 w\big{|}_{y=0},
\label{2.17}
\end{array}
         \right.\end{equation}
\begin{eqnarray}
(\partial _{t}-\epsilon^2\partial_x^2)(w\partial_x w)\big{|}_{y=0}=
(w\partial_x\partial_y^2w+\partial_xw\partial_y^2w-2w\partial_xw-2\epsilon^2\partial_x\partial_x^2w)\big{|}_{y=0}.
 \label{2.18}
\end{eqnarray}
It follows from $(\ref{2.16})-(\ref{2.18})$ that
\begin{eqnarray}
\begin{aligned}
\partial _{y}^{5}w\big{|}_{y=0}
&=C_n\sum\limits_{i=0}^{2}(\partial _{t}-\epsilon^2\partial_x^2)^i K+(2\partial_y w\partial_x\partial_yw+4w\partial_x\partial_y^2w-\partial_xw\partial_y^2w)\big{|}_{y=0} \\
&\quad-w\partial _{x}w \big{|}_{y=0}-2\epsilon^2\partial_xw\partial_x^2w\big{|}_{y=0}.
 \label{2.19}
\end{aligned}
\end{eqnarray}
For $n=3$,
\begin{eqnarray}
\partial _{y}^{7}w\big{|}_{y=0}
&=&(\partial _{t}-\epsilon^2\partial_x^2)\partial_y^{5}w\big{|}_{y=0}+\partial_y^{5}w\big{|}_{y=0}+\partial_y^5(u\partial _{x}w+v\partial _{y}w)\big{|}_{y=0}\nonumber \\
&=&C_n\sum\limits_{i=1}^{3}(\partial _{t}-\epsilon^2\partial_x^2)^iK+(\partial _{t}-\epsilon^2\partial_x^2)(2\partial_y w\partial_x\partial_yw+4w\partial_x\partial_y^2w-\partial_xw\partial_y^2w)\big{|}_{y=0}\nonumber \\
&\quad&-(\partial _{t}-\epsilon^2\partial_x^2)(w\partial_x w)\big{|}_{y=0}-(\partial _{t}-\epsilon^2\partial_x^2)2\epsilon^2\partial_xw\partial_x^2w\big{|}_{y=0}\nonumber\\
&\quad&+\partial_y^{5}w\big{|}_{y=0}+\partial_y^5(u\partial _{x}w+v\partial _{y}w)\big{|}_{y=0}.
 \label{2.20}
\end{eqnarray}
Since the first term and last four terms on the right-hand side  of $(\ref{2.20})$ are in the desired form,  we only need to calculate the second term on the right-hand side of $(\ref{2.20})$,
\begin{eqnarray}
&&(\partial _{t}-\epsilon^2\partial_x^2)(2\partial_y w\partial_x\partial_yw+4w\partial_x\partial_y^2w-\partial_xw\partial_y^2w)\big{|}_{y=0}\nonumber\\
&&= (-\partial _{x}\partial _{t}w\partial _{y}^{2}w-\partial _{x}w \partial _{y}^{2}\partial _{t}w+2\partial _{x}\partial _{y}\partial _{t}w\partial _{y}w+2\partial _{x}\partial _{y}w\partial _{y}\partial _{t}w
+4w\partial _{x}\partial _{y}^{2}\partial _{t}w +4\partial _{t}w \partial _{x}\partial _{y}^{2}w )\big{|}_{y=0}\nonumber\\
&&\quad-\epsilon^2\partial_x^2(2\partial_y w\partial_x\partial_yw+4w\partial_x\partial_y^2w-\partial_xw\partial_y^2w)\big{|}_{y=0},\label{2.21}
\end{eqnarray}
 here
 \begin{eqnarray}
\begin{aligned}
&-(\partial _{x}\partial _{t}w\partial _{y}^{2}w)\big{|}_{y=0}=-\partial _{y}^{2}w(\partial _{x}\partial _{y}^{2}w-\partial _{x}w+\epsilon^2\partial_x^3 w)\big{|}_{y=0}
 \label{2.22}
\end{aligned}
\end{eqnarray}
 \begin{eqnarray}
-\partial _{y}^{2}\partial _{t} w\big{|}_{y=0}
&=& -\partial _{y}^{2}(\partial _{y}^{2}w-u\partial _{x}w-v\partial _{y}w-w+\epsilon^2\partial_x^2w)\big{|}_{y=0}\nonumber\\
&=&-(\partial _{y}^{4}w-\partial _{y}^{2}w+2w\partial _{x}\partial _{y}w+\epsilon^2\partial _{y}^{2}\partial_x^2w)\big{|}_{y=0},
\end{eqnarray}
which implies
 \begin{eqnarray}
\begin{aligned}
\partial _{x}w\partial _{y}^{2}\partial _{t} w\big{|}_{y=0}&=\sum\limits_{j=0}^{4} \wedge_{\alpha}\partial _{x}  \partial _{y}^{j} w  \partial _{y}^{4-j} w
+\sum\limits_{j=0}^{2 } \wedge_{\alpha}\partial _{x}  \partial _{y}^{j} w  \partial _{y}^{2-j} w+\sum\limits_{\rho_{1}+\rho_{2}+\rho_{3}=3} \wedge_{\alpha} D^{\rho_{1}}w D^{\rho_{2}}w D^{\rho_{3}}w \big{|}_{y=0}\\
&\quad -\epsilon^2\partial _{x}w\partial _{y}^{2}\partial_x^2w\big{|}_{y=0}.
\end{aligned}
\end{eqnarray}
Since $ v=\partial _{x}\partial _{y}^{-1}w$,
\begin{eqnarray}
\partial _{x}\partial _{y}\partial _{t} w\big{|}_{y=0}
&=&\partial _{x}\partial _{y}(\partial _{y}^{2}w-u\partial _{x}w-v\partial _{y}w-w+\epsilon^2\partial_x^2w)\big{|}_{y=0}\nonumber\\
&=&\partial _{x}\partial _{y}^{3}w-\partial _{x}\partial _{y}w+\epsilon^2\partial_y\partial_x^3w-\partial _{x}(w\partial _{x}w+u\partial _{x}\partial _{y}w+\partial _{y}v\partial _{y}w+v\partial _{y}^2w)\big{|}_{y=0}\nonumber\\
&=&\partial _{x}\partial _{y}^{3}w-\partial _{x}\partial _{y}w+\epsilon^2\partial_y\partial_x^3w-\partial _{x}w\partial _{x}w-w\partial _{x}^2w-w\partial _{y}\partial _{x}w\big{|}_{y=0}
\end{eqnarray}
which implies
 \begin{eqnarray}
\begin{aligned}
&2\partial _{x}\partial _{y}\partial _{t}w\partial _{y} w\big{|}_{y=0} \\
&=\left(\sum\limits_{j=0}^{4} \wedge_{\alpha}\partial _{x}  \partial _{y}^{j} w  \partial _{y}^{4-j} w
+\sum\limits_{j=0}^{2 } \wedge_{\alpha}\partial _{x}  \partial _{y}^{j} w  \partial _{y}^{2-j} w+\sum\limits_{\rho_{1}+\rho_{2}+\rho_{3}=3} \wedge_{\alpha} D^{\rho_{1}}w D^{\rho_{2}}w D^{\rho_{3}}w \right)\big{|}_{y=0}\\
&\quad+2\epsilon^2\partial_y\partial_x^3w\partial_yw\big{|}_{y=0}.
\end{aligned}
\end{eqnarray}
Similarly,
 \begin{eqnarray}
\begin{aligned}
&\left(2\partial _{x}\partial _{y}w\partial _{y}\partial _{t}w
+4w\partial _{x}\partial _{y}^{2}\partial _{t}w +4\partial _{t}w \partial _{x}\partial _{y}^{2}w\right) \big{|}_{y=0}\\
&\in \left(\sum\limits_{j=0}^{4} \wedge_{\alpha}\partial _{x}  \partial _{y}^{j} w  \partial _{y}^{4-j} w
+\sum\limits_{j=0}^{2 } \wedge_{\alpha}\partial _{x}  \partial _{y}^{j} w  \partial _{y}^{2-j} w+\sum\limits_{\rho_{1}+\rho_{2}+\rho_{3}=3} \wedge_{\alpha} D^{\rho_{1}}w D^{\rho_{2}}w D^{\rho_{3}}w \right)\big{|}_{y=0}\\
&\quad+2\epsilon^2\partial _{x}\partial _{y}w\partial _{y}\partial _{x}^2w\big{|}_{y=0}+4\epsilon^2w\partial _{x}^3\partial _{y}^{2}w\big{|}_{y=0}+4\epsilon^2\partial _{x}^2w \partial _{x}\partial _{y}^{2}w\big{|}_{y=0}. \label{2.23}
\end{aligned}
\end{eqnarray}
Thus it follows from $(\ref{2.20})-(\ref{2.23})$ that
\begin{align}
\partial _{y}^{7}w\big{|}_{y=0}
&=\sum\limits_{i=0}^{3}(\partial _{t}-\epsilon^2\partial_x^2)^iK+(1+\epsilon^2\partial^2_x+\epsilon^4\partial^4_x)(w\partial _{x}w)\big{|}_{y=0}+\sum\limits_{j=0}^{0 } \partial _{x}  \partial _{y}^{0} w  \partial _{y}^{0-j} w\big{|}_{y=0}\nonumber\\
&\quad+\sum\limits_{j=0}^{2 } \wedge_{\alpha}\partial _{x}  \partial _{y}^{j} w  \partial _{y}^{2-j} w\big{|}_{y=0}+\sum\limits_{j=0}^{4 } \wedge_{\alpha}\partial _{x}  \partial _{y}^{j} w  \partial _{y}^{4-j} w\big{|}_{y=0} +\sum\limits_{\rho_{1}+\rho_{2}+\rho_{3}=3} \wedge_{\alpha} D^{\rho_{1}}w D^{\rho_{2}}w D^{\rho_{3}}w \big{|}_{y=0}\nonumber\\
&\quad+\sum\limits_{j_x=0}^{2 } \sum\limits_{j_y=0}^{2 }\wedge_{\alpha}\partial _{x}  \partial _{y}^{j_x}\partial _{y}^{j_y} w \partial _{y}^{2-j_x} \partial _{y}^{2-j_y} w\big{|}_{y=0}.
 \label{2.24}
\end{align}
We justify the formula (\ref{2.24}) for $k=3$.

Now, using the same algorithm, we are going to prove formula (\ref{2.005}) by induction on $k$. For notational convenience, we denote
\begin{eqnarray}
\mathcal{A}_k:=\bigg\{\sum\limits_{l=0}^{k-1}\epsilon^{2l}\sum\limits_{j=2}^{\max\{2,k-l\}}\sum\limits_{\rho \in A^j_{k,l}}C_{\wedge_{\alpha},k,l,\rho^1,\rho^2,...,\rho^j}\prod^{j}_{i=1}D^{\rho^i}w\big{|}_{y=0}\bigg\}.
\end{eqnarray}
Using this notation, we will prove $\partial_y^{2k+1}w\big{|}_{y=0}-(\partial_t-\epsilon^2\partial_x^2)^k(\partial_x  P-U) \in \mathcal{A}_k$. Assuming that formula (\ref{2.005}) holds for $k=n$, we will show that it also holds for $k=n+1$ as follows. Then we differentiate the vorticity equation with respect to $y$ $2n+1$ times to obtain
\begin{align}
\partial _{y}^{2n+3}w\big{|}_{y=0}&= \left\{(\partial _{t}-\epsilon^2\partial_x^2)\partial _{y}^{2n+1}w+\partial _{y}^{2n+1}(u\partial _{x}w+v\partial _{y}w)+\partial _{y}^{2n+1}w\right\}\big{|}_{y=0}\nonumber\\
&=\left\{(\partial _{t}-\epsilon^2\partial_x^2)\partial _{y}^{2n+1}w+\sum\limits_{j=0}^{2n } \wedge_{\alpha}\partial _{x}  \partial _{y}^{j} w  \partial _{y}^{2n-j} w+\partial _{y}^{2n+1}w\right\}\big{|}_{y=0}.
\label{2.006}
\end{align}
By routine checking, one may show that the last three terms of (\ref{2.006}) belong to
$\mathcal{A}_{k+1}$, so it only remains to deal with the term $(\partial _{t}-\epsilon^2\partial_x^2)\partial _{y}^{2n+1}w$.
Thanks to the induction hypothesis, there exist constants $C_{\wedge_{\alpha},n,l,\rho^1,\rho^2,...,\rho^j}$ such that
\begin{align}
\partial_y^{2n+1}w\big{|}_{y=0}&=C_n\sum\limits_{s=0}^{n}(\partial_t-\epsilon^2\partial^2_x)^n(\partial_x P -U)\nonumber\\
&\quad+\sum\limits_{l=0}^{k-1}\epsilon^{2l}\sum\limits_{j=2}^{\max\{2,n-l\}}\sum\limits_{\rho \in A^j_{n,l}}C_{\wedge_{\alpha},n,l,\rho^1,\rho^2,...,\rho^j}\prod^{j}_{i=1}D^{\rho^i}w\big{|}_{y=0}.
\end{align}
Thus we have, up to a relabeling of the indices $\rho^i$,
\begin{align}
(\partial_t-\epsilon^2\partial^2_x)\partial_y^{2n+1}w\big{|}_{y=0}&=C_n\sum\limits_{s=1}^{n+1}(\partial_t-\epsilon^2\partial^2_x)^{s}(\partial_x P -U)\nonumber\\
&\quad+\sum\limits_{l=0}^{n-1}\epsilon^{2l}\sum\limits_{j=2}^{\max\{2,n-l\}}\sum\limits_{\rho \in A^j_{n,l}} \widetilde{C}_{\wedge_{\alpha},n,l,\rho^1,\rho^2,...,\rho^j}(\partial_t-\epsilon^2\partial^2_x)D^{\rho^1}w\prod^{j}_{i=2}D^{\rho^i}w\big{|}_{y=0}\nonumber\\
&\quad-\sum\limits_{l=0}^{n-1}\epsilon^{2l+2}\sum\limits_{j=2}^{\max\{2,n-l\}}\sum\limits_{\rho \in A^j_{n,l}}\widetilde{\widetilde{C}}_{\wedge_{\alpha},n,l,\rho^1,\rho^2,...,\rho^j}\partial_x D^{\rho^1}w\partial_x D^{\rho^2}w\prod^{j}_{i=3}D^{\rho^i}w\big{|}_{y=0},
\label{2.007}
\end{align}
where $\widetilde{C}_{\wedge_{\alpha},n,l,\rho^1,\rho^2,...,\rho^j}$ and $\widetilde{\widetilde{C}}_{\wedge_{\alpha},n,l,\rho^1,\rho^2,...,\rho^j}$ are some new constants depending on $C_{\wedge_{\alpha},n,l,\rho^1,\rho^2,...,\rho^j}$. It is worth noting that the last term on the right-hand side of (\ref{2.007})
belongs to $\mathcal{A}_{n+1}$, so it remains to check whether the second term on right-hand
side of (\ref{2.007}) also belongs to $\mathcal{A}_{n+1}$.

Differentiating the vorticity equation $(\ref{2.0002})_1$ $\rho_1^1$ times and $\rho_2^1$ times with respect to $x$  and $y$, respectively and then calculating at $y=0$, we have, by denoting $e_2=(0,1)$,
\begin{align}
(\partial_t-\epsilon^2\partial_x^2)D^{\rho^1}w \big{|}_{y=0}=&-\sum_{\substack{\beta \le \rho^1 \\
\beta_2 \geq 1}}\binom{\rho^1}{\beta} D^{\beta-e_2} w \partial_x D^{\rho_1-\beta} w \big{|}_{y=0}+\sum_{\substack{\beta \le \rho^1 \\
\beta_2 \geq 2}}\binom{\rho^1}{\beta} \partial_xD^{\beta-2e_2} w \partial_y D^{\rho_1-\beta} w \big{|}_{y=0}\nonumber\\
&+\partial_y^2 D^{\rho^1}w \big{|}_{y=0}+ D^{\rho^1}w \big{|}_{y=0}.
\label{2.0027}
\end{align}
Using (\ref{2.0027}), one may justify by a routine counting of indices that the second term
on the right-hand side of (\ref{2.007}) belongs to $\mathcal{A}_{n+1}$.
This thus completes the proof.
\end{proof}
\begin{Lemma}\label{y2.2}
Let $s \geq 4$ be an even integer, $\gamma\geq 1$, $\sigma>\gamma+\frac{1}{2}$ and $\delta\in (0,1)$. If  $w\in H^{s,\gamma}_{\sigma,\delta}$  solves $(\ref{2.0001})$  and $(\ref{2.0002})$, we have the following  estimates:\\
(i) when $|\alpha|\leq s-1$,
\begin{eqnarray}
\bigg|\int_{\mathbb{T} } D^{\alpha}w\partial _{y}D^{\alpha}wdx \big{|} _{y=0}\bigg|\leq \frac{1}{12}\|(1+y)^{\gamma+\alpha_2+1}\partial^2_y D^\alpha w\|_{L^2}^2+C\|w\|^2_{H^{s,\gamma}_g},
\end{eqnarray}
(ii) when $|\alpha|=s$, $\alpha_2$ is even,
\begin{align}
\bigg|\int_{\mathbb{T} } D^{\alpha}w\partial _{y}D^{\alpha}wdx \big{|} _{y=0}\bigg|\leq &\frac{1}{12}\|(1+y)^{\gamma+\alpha_2}\partial_y D^\alpha w\|^2_{L^2}+C_{s,\gamma,\sigma,\delta}(1+\|w\|^2_{H^{s,\gamma}_g})^{s-2}\|w\|^2_{H^{s,\gamma}_g}\nonumber\\
&+C_s \sum\limits_{l=0}^{\frac{s}{2}}\|\partial^{l}_t(\partial_xP-U)\|_{H^{s-2l}(\mathbb{T})}^{2},
\end{align}
(iii) when $|\alpha|=s$, $\alpha_2$ is odd,
\begin{align}
\bigg|\int_{\mathbb{T} } D^{\alpha}w\partial _{y}D^{\alpha}wdx \big{|} _{y=0}\bigg|\leq& \frac{1}{12}\|(1+y)^{\gamma+\alpha_2+1}\partial_x^{\alpha_1-1}\partial_y^{\alpha_2+2} w\|^2_{L^2}+C_{s,\gamma,\sigma,\delta}(1+\|w\|^2_{H^{s,\gamma}_g})^{s-2}\|w\|^2_{H^{s,\gamma}_g}\nonumber\\
&+C_s \sum\limits_{l=0}^{\frac{s}{2}}\|\partial^{l}_t(\partial_xP-U)\|_{H^{s-2l}(\mathbb{T})}^{2}.
\end{align}
\begin{proof}
\emph{Case 1.} When $|\alpha|\leq s-1$,  using the following trace estimate
\begin{eqnarray}
\int_{\mathbb{T} } |f|dx \big{|}_{y=0} \leq C\left(\int^1_0\int_{\mathbb{T}}|f|dxdy+\int^1_0\int_{\mathbb{T}}|\partial_y f|dxdy\right),
\end{eqnarray}
we have
\begin{eqnarray}
\bigg|\int_{\mathbb{T} } D^{\alpha}w\partial _{y}D^{\alpha}wdx \big{|} _{y=0}\bigg|\leq \frac{1}{12}\|(1+y)^{\gamma+\alpha_2+1}\partial^2_y D^\alpha w\|_{L^2}^2+C\|w\|^2_{H^{s,\gamma}_g}.
\end{eqnarray}
\emph{Case 2.} When $|\alpha|=s$, $\alpha_2=2k$ for some $k\in \mathbb{N}$, we can apply Lemma \ref{y2.1} to $\partial_y D^\alpha  | _{y=0}$ to obtain
\begin{align}
\int_{\mathbb{T} } D^{\alpha}w\partial _{y}D^{\alpha}wdx \big{|} _{y=0}=&C_k\sum\limits_{s=0}^{k}\int_{\mathbb{T} } D^{\alpha}w(\partial_t-\epsilon^2\partial_x^2)^k\partial_x^{\alpha_1}(\partial_xP-U)dx \big{|} _{y=0}\nonumber\\
&+\sum\limits_{l=0}^{k-1}\epsilon^{2l}\sum\limits_{j=2}^{\max\{2,k-l\}}\sum\limits_{\rho \in A^j_{k,l}}C_{\wedge_{\alpha},k,l,\rho^1,\rho^2,...,\rho^j}\int_{\mathbb{T} }D^\alpha w \partial_x^{\alpha_1}(\prod^{j}_{i=1}D^{\rho^i}w)dx\big{|}_{y=0}.
\end{align}
Then we again apply the simple trace estimate to control the boundary integral as follows
\begin{align}
\bigg|\int_{\mathbb{T} } D^{\alpha}w\partial _{y}D^{\alpha}wdx \big{|} _{y=0}\bigg|\leq& \frac{1}{12}\|(1+y)^{\gamma+\alpha_2}\partial_y D^\alpha w\|^2_{L^2}+C_{s,\gamma,\sigma,\delta}(1+\|w\|_{H_g^{s,\gamma}})^{s-2}\|w\|^2_{H^{s,\gamma}_g}\nonumber\\
&+C_s \sum\limits_{l=0}^{\frac{s}{2}}\|\partial^{l}_t(\partial_xP-U)\|_{H^{s-2l}(\mathbb{T})}^{2}.
\end{align}
\emph{Case 3.} When $|\alpha|=s$, $\alpha_2=2k+1$ for some $k\in \mathbb{N}$, using integration by parts in the $x$-variable, we have
\begin{eqnarray}
\int_{\mathbb{T} } D^{\alpha}w\partial _{y}D^{\alpha}wdx \big{|} _{y=0}=-\int_{\mathbb{T}}\partial_x D^\alpha w \partial_x^{\alpha_1-1}\partial_y^{\alpha_2+1}wdx \big{|}_{y=0}.
\end{eqnarray}
Noting that the term $\partial_x D^\alpha w\big{|}_{y=0}=\partial_x^{\alpha_1+1}\partial_y^{2k+1}w \big{|} _{y=0}$ has an odd number of $y$ derivatives, then we get
\begin{align}
\bigg|\int_{\mathbb{T} } D^{\alpha}w\partial _{y}D^{\alpha}wdx \big{|} _{y=0}\bigg|\leq& \frac{1}{12}\|(1+y)^{\gamma+\alpha_2+1}\partial_x^{\alpha_1-1}\partial_y^{\alpha_2+2} w\|^2_{L^2}+C_{s,\gamma,\sigma,\delta}(1+\|w\|_{H_g^{s,\gamma}})^{s-2}\|w\|^2_{H^{s,\gamma}_g}\nonumber\\
&+C_s \sum\limits_{l=0}^{\frac{s}{2}}\|\partial^{l}_t(\partial_xP-U)\|_{H^{s-2l}(\mathbb{T})}^{2}.
\end{align}
\end{proof}
\end{Lemma}

\begin{Proposition}\label{p2.1}
Let $s \geq 4$ be an even integer, $\gamma\geq 1$, $\sigma>\gamma+\frac{1}{2}$ and $\delta\in (0,1)$. If  $w\in H^{s,\gamma}_{\sigma,\delta}$  solves $(\ref{2.0001})$  and $(\ref{2.0002})$, then we have the following uniform (in $\epsilon$) estimate:
\begin{align}
&\frac{1}{2}\frac{d}{dt}\sum\limits_{\substack{ |\alpha|\leq s\\ \alpha_{1}\leq s-1}}\|(1+y)^{ \gamma+ \alpha_{2}}D^{\alpha}w\|_{L^{2}}^{2}+\sum\limits_{\substack{ |\alpha|\leq s\\ \alpha_{1}\leq s-1}}\|(1+y)^{ \gamma+ \alpha_{2}}D^{\alpha}w\|_{L^{2}}^{2}
\nonumber \\
&\leq -\frac{1}{2}\sum\limits_{\substack{ |\alpha|\leq s\\ \alpha_{1}\leq s-1}}\|(1+y)^{ \gamma+ \alpha_{2}}\partial _{y}D^{\alpha}w\|_{L^{2}} ^{2}-\epsilon^2\sum\limits_{\substack{ |\alpha|\leq s\\ \alpha_{1}\leq s-1}} \|(1+y)^{\gamma+ \alpha_{2}}\partial _{x}D^{\alpha}w\|_{L^{2}}^2\nonumber\\
&\quad \quad+C_{s, \gamma,\sigma,\delta  }\|w\|_{H^{s,\gamma}_{g}}^2(\|\partial_{x}^{s}U\|_{{L^\infty}(\mathbb{T})}^2+\|w\|_{H^{s,\gamma}_{g}}^2)+C_{s,\gamma,\sigma,\delta}(1+\|w\|_{H_g^{s,\gamma}})^{s-2}\|w\|^2_{H^{s,\gamma}_g}\nonumber\\
&\quad \quad+C_s \sum\limits_{l=0}^{\frac{s}{2}}\|\partial^{l}_t(\partial_xP-U)\|_{H^{s-2l}(\mathbb{T})}^{2}.
\label{2.002}
\end{align}
\end{Proposition}
\begin{proof}
Applying the operator $D^{\alpha}$ on $(\ref{2.0001})_{1}$ with $\alpha=(\alpha_{1},\alpha_{2}),~|\alpha|\leq s, ~\alpha_{1}\leq s-1$,
\begin{eqnarray}
(\partial _{t} +u\partial _{x} +v\partial _{y}+1-\epsilon^2\partial _{x}^{2}-\partial _{y}^{2})D^{\alpha}w
=- \sum\limits_{0<\beta\leq \alpha}\binom{ \alpha}{\beta} \left\{D^{\beta} u\partial _{x}D^{\alpha-\beta}w + D^{\beta}v\partial _{y}D^{\alpha-\beta}w\right\} .
 \label{2.1}
\end{eqnarray}
Multiplying $(\ref{2.1})$ by $(1+y)^{2\gamma+2\alpha_{2}}D^{\alpha}w$ and integrating it over $\mathbb{T}\times \mathbb{R}_{+}$ yield
\begin{align}
&\frac{1}{2}\frac{d}{dt}\|(1+y)^{ \gamma+ \alpha_{2}}D^{\alpha}w\|_{L^{2}}^{2}+\|(1+y)^{ \gamma+ \alpha_{2}}D^{\alpha}w\|_{L^{2}}^{2}
+\|(1+y)^{ \gamma+ \alpha_{2}}\partial _{y}D^{\alpha}w\|_{L^{2}}^{2}+ \epsilon^2\|(1+y)^{\gamma+ \alpha_{2}}\partial _{x}D^{\alpha}w\|_{L^{2}}^2   \nonumber\\
&=-\int_{\mathbb{T} } D^{\alpha}w\partial _{y}D^{\alpha}wdx | _{y=0}-(2\gamma+2\alpha_{2}) \iint(1+y)^{2\gamma+2\alpha_{2}-1} D^{\alpha}w\partial _{y}   D^{\alpha}w\nonumber\\
&\quad+(\gamma+\alpha_{2}) \iint(1+y)^{2\gamma+2\alpha_{2}-1} v  | D^{\alpha}w|^{2}  \nonumber\\
&\quad - \sum\limits_{0<\beta\leq \alpha}\binom{ \alpha}{\beta} \iint(1+y)^{2\gamma+2\alpha_{2}}D^{\alpha}w  \left\{D^{\beta} u\partial _{x}D^{\alpha-\beta}w + D^{\beta}v\partial _{y}D^{\alpha-\beta}w\right\} . \label{2.2}
\end{align}
Indeed, $(\ref{2.2})$ can be obtained by integration by parts and using the boundary condition.
We need to estimate  $(\ref{2.2})$ term by term. Obviously, the first term on the right-hand side of $(\ref{2.2})$ follows from Lemmas \ref{y2.1}-\ref{y2.2}.

Secondly,  using H$\ddot{o}$lder inequality, we have
\begin{eqnarray}
&&\left|(2\gamma+2\alpha_{2}) \iint(1+y)^{2\gamma+2\alpha_{2}-1} D^{\alpha}w\partial _{y}   D^{\alpha}w  \right|\nonumber \\
 &&\leq (2\gamma+2\alpha_{2}) \|\frac{1}{1+y }  \|_{L^{\infty}} \|(1+y)^{\gamma+ \alpha_{2}}D^{\alpha}w \|_{L^{2}}
  \|(1+y)^{\gamma+ \alpha_{2}}\partial_{y} D^{\alpha}w \|_{L^{2}} \nonumber \\
&&\leq  C_{ s,\gamma } \|w\|_{H^{s,\gamma}_{g}}^2+\frac{1}{4}\|(1+y)^{\gamma+ \alpha_{2}}\partial_{y} D^{\alpha}w \|_{L^{2}}^{2} ,  \label{2.5}
\end{eqnarray}
and using Lemma \ref{y4.4},
\begin{eqnarray}
\left|   (\gamma+\alpha_{2}) \iint(1+y)^{2\gamma+2\alpha_{2}-1} v  | D^{\alpha}w|^{2}  \right|
 &\leq&  \|\frac{v}{1+y}\|_{L^{\infty}}  \|(1+y)^{\gamma+ \alpha_{2}}D^{\alpha}w \|_{L^{2}}^{2} \nonumber \\
& \leq &  C_{s, \gamma,\sigma,\delta  }(\|w\|_{H^{s,\gamma}_{g}}+\|\partial_{x}^{s}U\|_{{L^2}(\mathbb{T})})\|w\|_{H^{s,\gamma}_{g}}^2.  \label{2.6}
\end{eqnarray}
Lastly, noticing the fact $\partial _{y}v=-\partial _{x}u, ~\partial _{y}u=w$, it follows that the last term on the right-hand side of  $(\ref{2.2})$ has the following three cases, \\ (i) the first case:
$$ J_{1}= \iint(1+y)^{2\gamma+2\alpha_{2}}D^{\alpha}w  \partial _{x}^{\eta}vD^{k}w ,\ \   \eta\in [1,s-1].  $$
(a) for $\eta=s-1$, we get
\begin{align}
|J_{1}| &= \left| \iint(1+y)^{2\gamma+2\alpha_{2}}D^{\alpha}w  \partial _{x}^{s-1}v\partial _{y}^{2}w  \right|\nonumber \\
 &\leq   \|(1+y)^{\gamma+ \alpha_{2}}D^{\alpha}w \|_{L^{2}}   \|\frac{\partial _{x}^{s-1}v +y\partial_x^sU} {1+y}\|_{L^{2}}
  \|(1+y)^{\gamma+k_2}D ^{k }w\|_{L^{\infty}}   \nonumber \\
  & \quad+ \|(1+y)^{\gamma+ \alpha_{2}}D^{\alpha}w \|_{L^{2}}   \|\partial_x^s U\|_{{L^\infty}(\mathbb{T})}
  \|(1+y)^{\gamma+k_2} D ^{k }w\|_{L^{\infty}}   \nonumber \\
&\leq  C_{s, \gamma,\sigma,\delta  }\|w\|_{H^{s,\gamma}_{g}} ^2(\|w\|_{H^{s,\gamma}_{g}}+\|\partial_{x}^{s}U\|_{{L^\infty}(\mathbb{T})});
\label{2.7}
\end{align}
(b) for $\eta=1,2,\cdots,s-2$, we derive
\begin{align}
|J_{1}|
&\leq    \|(1+y)^{\gamma+ \alpha_{2}}D^{\alpha}w \|_{L^{2}}   \|\frac{\partial _{x}^{\eta}v } {1+y}\|_{L^{\infty}}
  \|(1+y)^{\gamma+k_{2}}D^{k}w\|_{L^{2}}   \nonumber \\
&\leq  C_{s, \gamma,\sigma,\delta  }\|w\|_{H^{s,\gamma}_{g}}^2 (\|w\|_{H^{s,\gamma}_{g}}+\|\partial_{x}^{s}U\|_{{L^2}(\mathbb{T})}).
\label{2.8}
\end{align}
(ii) the second case£º
$$ J_{2}= \iint(1+y)^{2\gamma+2\alpha_{2}}D^{\alpha}w  \partial _{x}^{\eta}uD^{k}w ,\ \ \eta\in [1,s ].  $$
(a) for $\eta=s $, we get
\begin{align}
|J_{2}|& = \left| \iint(1+y)^{2\gamma+2\alpha_{2}}D^{\alpha}w  \partial _{x}^{s }u\partial _{y} w  \right|\nonumber\\
 &\leq   \|(1+y)^{\gamma+ \alpha_{2}}D^{\alpha}w \|_{L^{2}}   \|\partial _{x}^{s }(u-U)\|_{L^{2}}
  \|(1+y)^{\gamma+k_2} D^k w\|_{L^{\infty}}   \nonumber \\
  &\quad+\|(1+y)^{\gamma+ \alpha_{2}}D^{\alpha}w \|_{L^{2}}   \|\partial _{x}^{s }U \|_{{L^\infty}(\mathbb{T})}
  \|(1+y)^{\gamma+k_2} D^k w\|_{L^{2}}   \nonumber \\
&\leq  C_{s, \gamma,\sigma,\delta  }\|w\|_{H^{s,\gamma}_{g}}^2 (\|w\|_{H^{s,\gamma}_{g}}+\|\partial_{x}^{s}U\|_{{L^\infty}(\mathbb{T})});
\label{2.9}
\end{align}
(b) for $\eta=1,2,\cdots,s-1$, we derive
\begin{align}
|J_{2}|
&\leq   \|(1+y)^{\gamma+ \alpha_{2}}D^{\alpha}w \|_{L^{2}}   \| \partial _{x}^{\eta}u   \|_{L^{\infty}}
  \|(1+y)^{\gamma+k_{2}}D^{k}w\|_{L^{2}}   \nonumber \\
&\leq  C_{s, \gamma,\sigma,\delta  }\|w\|_{H^{s,\gamma}_{g}}^2 (\|w\|_{H^{s,\gamma}_{g}}+\|\partial_{x}^{s}U\|_{{L^2}(\mathbb{T})}).
\label{2.10}
\end{align}
(iii) the last case£º
$$ J_{3}= \iint(1+y)^{2\gamma+2\alpha_{2}}D^{\alpha}w  D^{\theta}w D^{k}w ,\ \ \theta\in [0,s-1 ],  \theta+k=s. $$
We obtain
\begin{align}
|J_{3}|
&\leq  C_{s, \gamma,\sigma,\delta  }\|w\|_{H^{s,\gamma}_{g}} (\|w\|_{H^{s,\gamma}_{g}}+\|\partial_{x}^{s}U\|_{{L^2}(\mathbb{T})})\|(1+y)^{\gamma+ \alpha_{2}}D^{\alpha}w \|_{L^{2}}\nonumber \\
 &\leq C_{s, \gamma,\sigma,\delta  }\|w\|_{H^{s,\gamma}_{g}}^{2}(\|w\|_{H^{s,\gamma}_{g}}+\|\partial_{x}^{s}U\|_{{L^2}(\mathbb{T})})^2+\frac{1}{12} \|(1+y)^{\gamma+ \alpha_{2}}D^{\alpha}w \|_{L^{2}}  ^{2},
\label{2.12}
\end{align}
Hence, combining $(\ref{2.5})$-$(\ref{2.12})$, $(\ref{2.2})$ reduces to inequality (\ref{2.002}).
\end{proof}

\subsection{Estimates on $g_{s}$ }
In this subsection, we will estimate the norm of $D^{\alpha}w$ with weight  $(1+y)^{\gamma+\alpha_{2}}$, here $\alpha=(s, 0)$, i.e., $(1+y)^{\gamma}g_{s}$, because  we have estimated the norm of $(1+y)^{\gamma+\alpha_{2}} D^{\alpha}w$,    $\alpha=(\alpha_{1}, \alpha_{2})$, $\alpha_{1}\leq s-1$ in subsection 2.1.
The evolution equations for $w$ and $u-U$ as follows:
\begin{equation}\left\{
\begin{array}{ll}
(\partial _{t} +u\partial _{x} +v\partial _{y}+1-\partial _{y}^{2}-\epsilon^2\partial _{x}^{2}) w
=0,\\
(\partial _{t} +u\partial _{x} +v\partial _{y}+1-\partial _{y}^{2}-\epsilon^2\partial _{x}^{2}) (u-U)
=-  (u-U)\partial _{x} U  .
\end{array}
 \label{3.01}         \right.\end{equation}

\begin{Proposition}\label{p2.2}
Under the same assumption of Proposition \ref{p2.1}, we have the following estimate:
\begin{align}
&\frac{d}{dt}\|(1+y)^{ \gamma}g_{s}\|_{L^{2}}^{2}+\|(1+y)^{ \gamma}g_{s}\|_{L^{2}}^{2} \nonumber\\
&\leq -\frac{1}{2}\epsilon^2\|(1+y)^{\gamma}\partial _{x}g_{s}\|_{L^{2}} ^{2} -\frac{1}{2}\|(1+y)^{\gamma}\partial _{y}g_{s}\|_{L^{2}} ^{2}+\|w\|_{H^{s,\gamma}_{g}}^2 \nonumber\\
&\quad+ C\|\partial_{x}^{s}(\partial_xP-U)\|_{{L^2}(\mathbb{T})}^2+C_{ \gamma,\delta  }\|\partial_{x}^{s}U\|_{{L^{2}}(\mathbb{T})}^2 \|w\|_{H^{s,\gamma}_{g}}^2\nonumber\\
&\quad+ C_{s, \gamma,\sigma,\delta  }\left(1+\|w\|_{H^{s,\gamma}_{g}}+\|\partial_{x}^{s}U\|_{{L^\infty}(\mathbb{T})}\right)
\left(\|w\|_{H^{s,\gamma}_{g}}+\|\partial_{x}^{s+1}U\|_{{L^\infty}(\mathbb{T})}\right)\|w\|_{H^{s,\gamma}_{g}}.
\label{3.00003}
\end{align}
\end{Proposition}
\begin{proof}
Applying the operator $\partial_{x}^{s}$ on $(\ref{3.01})_{1}$ and $(\ref{3.01})_{2}$ respectively, we obtain the following  equations
\begin{equation}\left\{
\begin{array}{ll}
(\partial _{t} +u\partial _{x} +v\partial _{y}+1-\partial _{y}^{2}-\epsilon^2\partial _{x}^{2})\partial_{x}^{s}w+\partial_{x}^{s}v\partial _{y}w
=- \sum\limits_{0\leq j<s}\binom{s}{j} \partial _{x}^{s-j}u\partial _{x}^{j+1}w \\
\quad- \sum\limits_{1\leq j<s}\binom{s}{j} \partial _{x}^{s-j}v\partial _{y}\partial _{x}^{j }w  ,\\
(\partial _{t} +u\partial _{x} +v\partial _{y}+1-\partial _{y}^{2}-\epsilon^2\partial _{x}^{2})\partial_{x}^{s}(u-U)+\partial_{x}^{s}v\partial _{y}u
=- \sum\limits_{0\leq j<s}\binom{s}{j} \partial _{x}^{s-j}u\partial _{x}^{j+1}(u-U) \\
\quad- \sum\limits_{1\leq j<s}\binom{s}{j} \partial _{x}^{s-j}v\partial _{y}\partial _{x}^{j }u- \sum\limits_{0\leq j\leq s}\binom{s}{j} \partial _{x}^{j}(u-U)\partial _{x}^{s-j+1}U.
 \end{array}  \label{3.1}
\right.\end{equation}
To overcome the difficult term  $\partial_{x}^{s}v$, we use the cancellation property as usual in \cite{[2]}. Indeed,  we  subtract $\frac{\partial _{y}w}{w}\times (\ref{3.1})_{2}$ from $ (\ref{3.1})_{1}$, and letting  $a(t,x,y)=\frac{\partial _{y}w}{w}$, it follows
\begin{align}
&
 (\partial _{t} +u\partial _{x} +v\partial _{y}+1-\partial _{y}^{2})g_{s} +\partial_{x}^{s}(u-U) (\partial _{t} +u\partial _{x} +v\partial _{y}+1-\partial _{y}^{2})a    \nonumber\\
&= 2\epsilon^2\partial_x^{s+1}(u-U)\partial_x a+2\partial_{x}^{s}w\partial _{y}a - \sum\limits_{j=0}^{s-1}\binom{s}{j} g_{j+1}\partial _{x}^{s-j}u
- \sum\limits_{j=1}^{s-1}\binom{s}{j}\partial _{x}^{s-j}v \{\partial _{y}\partial _{x}^{j }w-a \partial _{y}\partial _{x}^{j }u \}\nonumber\\
&\quad+a\sum\limits_{j=0}^s\binom{s}{j} \partial _{x}^{j}(u-U)\partial _{x}^{s-j+1}U . \label{3.2}
\end{align}
Applying the operator $\partial_{y}$ on $(\ref{2.0002})_{1}$ yields
\begin{eqnarray}
(\partial _{t} +u\partial _{x} +v\partial _{y}+1 )\partial _{y}w=2\epsilon^2\partial_x\partial_y w+\partial _{y}^{3}w-w\partial _{x}w+\partial _{x}u\partial _{y}w.  \label{3.3}
\end{eqnarray}
Using $(\ref{3.3})$ and equation $(\ref{2.0002})_{1}$, we get
\begin{align}
 (\partial _{t} +u\partial _{x} +v\partial _{y}+1 )a=&\frac{(\partial _{t} +u\partial _{x} +v\partial _{y}+1 )\partial _{y}w}{w} -
\frac{\partial _{y}w(\partial _{t} +u\partial _{x} +v\partial _{y}+1 ) w}{w^{2}}   \nonumber\\
=&\epsilon^2\frac{\partial_x^2\partial_y w}{w}-\epsilon^2a\frac{\partial_x^2w}{w}+\frac{\partial _{y}^{3}w}{w}-a\frac{\partial _{y}^{2}w}{w}-  g_1+a \partial _{x}U,
\label{3.4}
\end{align}
and
\begin{equation}\left\{
\begin{array}{ll}
  \partial _{y}a=\frac{ \partial _{y}^{2}w}{w} -\frac{\partial _{y}w \partial _{y} w}{w^{2}},   \\
 \partial _{y}^{2}a=\frac{\partial _{y}^{3}w}{w}-a\frac{\partial _{y}^{2}w}{w}- 2a \partial _{y}a .
\end{array}
 \label{3.5}         \right.\end{equation}
Now inserting $(\ref{3.4})$ and $(\ref{3.5})_{2}$ into $(\ref{3.2})$ yields
\begin{align}
&
 (\partial _{t} +u\partial _{x} +v\partial _{y}+1-\partial _{y}^{2}-\epsilon^2 \partial_x^2)g_{s}  \nonumber\\
&= 2\epsilon^2\big\{\partial_x^{x+1}(u-U)-\frac{\partial_xw}{w}\partial_x^2(u-U)\big\}\partial_x a+2g_{s}\partial _{y}a -g_1\partial _{x}^{s}U- \sum\limits_{j=1}^{s-1}\binom{s}{j} g_{j+1}\partial _{x}^{s-j}u
\nonumber\\
&\quad- \sum\limits_{j=1}^{s-1}\binom{s}{j}\partial _{x}^{s-j}v \{\partial _{y}\partial _{x}^{j }w-a \partial _{y}\partial _{x}^{j }u \}+a\sum\limits_{j=0}^{s-1}\binom{s}{j} \partial _{x}^{j}(u-U)\partial _{x}^{s-j+1}U . \label{3.6}
\end{align}
 Multiplying $(\ref{3.6})$ by $(1+y)^{2\gamma}g_{s}$, and integrating it over $\mathbb{T}\times \mathbb{R}_{+}$, we arrive at
\begin{align}
&\frac{1}{2}\frac{d}{dt}\|(1+y)^{ \gamma}g_{s}\|_{L^{2}}^{2}+\|(1+y)^{ \gamma}g_{s}\|_{L^{2}}^{2}+\|(1+y)^{\gamma}\partial _{y}g_{s}\|_{L^{2}} ^{2}+\epsilon^2\|(1+y)^{\gamma}\partial _{x}g_{s}\|_{L^{2}} ^{2} \nonumber \\
&=2\epsilon^2\iint (1+y)^{2\gamma}g_{s}\big\{\partial_x^{x+1}(u-U)-\frac{\partial_xw}{w}\partial_x^2(u-U)\big\}\partial_x a\nonumber\\
&\quad+\int_{\mathbb{T} } g_{s}\partial _{y}g_{s}dx | _{y=0}-2\gamma \iint(1+y)^{2\gamma-1}g_{s}\partial _{y}  g_{s}
+\gamma \iint(1+y)^{2\gamma-1} v  |g_{s}|^{2}  + 2\iint(1+y)^{2\gamma}|g_{s}|^{2}\partial _{y}a\nonumber\\
&\quad - \sum\limits_{j=1}^{s-1}\binom{s}{j} \iint(1+y)^{2\gamma}g_{s} g_{j+1}\partial _{x}^{s-j}u
- \sum\limits_{j=1}^{s-1}\binom{s}{j}\iint(1+y)^{2\gamma}g_{s}\partial _{x}^{s-j}v \{\partial _{y}\partial _{x}^{j }w-a \partial _{x}^{j }w \}\nonumber\\
&\quad-\iint(1+y)^{2\gamma}g_{s}g_1\partial _{x}^{s}U+\sum\limits_{j=0}^{s-1}\binom{s}{j}\iint(1+y)^{2\gamma}g_{s}a \partial _{x}^{j}(u-U)\partial _{x}^{s-j+1}U, \label{3.7}
\end{align}
which can be obtained by integration by parts and using the boundary condition.
We estimate  $(\ref{3.7})$ term by term. Firstly $w \in C([0,T];H^{s+4,\gamma}_{\sigma,\delta})$, it follows from the  definition of $H^{s+4,\gamma}_{\sigma,\delta}$ that $(1+y)^\sigma w \geq \delta$ and $|(1+y)^{\sigma+\alpha}D^\alpha w|\leq \delta^{-1}$ for all $|\alpha| \leq 2$. Thus, we have $\|(1+y)\partial_x a\|_{L^\infty} \leq \delta^{-2}+\delta^{-4}$ and $\|\frac{\partial_x w}{w}\|_{L^\infty} \leq \delta^{-2}$, and hence
\begin{align}
&2\epsilon^2\iint (1+y)^{2\gamma}g_{s}\big\{\partial_x^{x+1}(u-U)-\frac{\partial_xw}{w}\partial_x^2(u-U)\big\}\partial_x a\nonumber\\
&\leq2\epsilon^2 C_{\delta}\|(1+y)^{\gamma}g_s\|_{L^2}(\|(1+y)^{\gamma-1}\partial_x^{s+1}(u-U)\|_{L^2}+\|(1+y)^{\gamma-1}\partial_x^s(u-U)\|_{L^2}).
\end{align}
In addition, by using Lemma \ref{y4.1} and the following estimate
\begin{equation}\left\{
\begin{array}{ll}
 |w|^{-1}\leq \delta^{-1}(1+y)^{\sigma},   \\
 |\partial _{y}  w| \leq \delta^{-1}(1+y)^{-\sigma-1},  \ \
| \partial _{y}^{2}w| \leq \delta^{-1}(1+y)^{-\sigma-2},
\end{array}
 \label{2.60}         \right.\end{equation}
we have
\begin{align}
\|(1+y)^{\gamma-1}\partial_x^{s+1}(u-U)\|_{L^2} &\leq C_{\gamma, \sigma, \delta}\|(1+y)^{\gamma-\sigma-1}\frac{\partial_x^{s+1}(u-U)}{w}\|_{L^2}\nonumber\\
&\leq C_{\gamma, \sigma, \delta}\left(\|\partial_x^{s+1}U\|_{L^{2}(\mathbb{T})}+\|(1+y)^{\gamma}\partial_x g_s\|_{L^2}+\|(1+y)^{\gamma-1}\partial_x^{s} (u-U)\|_{L^2}\right),
\end{align}
and
\begin{eqnarray}
\begin{aligned}
\|(1+y)^{\gamma-1}\frac{\partial_x w}{w}\partial_x^{s}(u-U)\|_{L^2} \leq C_{\gamma, \sigma, \delta}\|(1+y)^{\gamma-1}\partial_x^{s} (u-U)\|_{L^2}.
\end{aligned}
\end{eqnarray}
Then, we can obtain
\begin{align}
&\bigg|2\epsilon^2\iint (1+y)^{2\gamma}g_{s}\big\{\partial_x^{x+1}(u-U)-\frac{\partial_xw}{w}\partial_x^2(u-U)\big\}\partial_x a\bigg|\nonumber\\
& \leq\frac{1}{2} \epsilon^2 \|(1+y)^\gamma \partial_x g_s\|_{L^2}^2 + \epsilon^2 C_{s,\gamma, \sigma, \delta}\left(\|w\|_{H^{s,\gamma}_g}+\|\partial_x^{s+1}U\|_{L^2(\mathbb{T})}\right)\|w\|_{H^{s,\gamma}_g}^2.
\end{align}
Secondly,  we have the fact that
\begin{align}
\partial _{y}g_{s} \big{|} _{y=0}&=\left(\partial_{x}^{s}\partial _{y}w-\frac{\partial _{y}w}{w} \partial_{x}^{s}w -\frac{\partial^{2} _{y}w}{w} \partial_{x}^{s}(u-U) +\frac{\partial_{y}w\partial_{y}w}{w^{2}} \partial_{x}^{s}(u-U) \right) \big{|} _{y=0} \nonumber\\
&=\left(\partial_{x}^{s}(\partial _{x}p-U)-ag_{s} +\frac{\partial^{2} _{y}w}{w} \partial_{x}^{s}U\right) \big{|} _{y=0}.  \label{3.10}
\end{align}
Then according to the trace theorem, we get
\begin{eqnarray}
&&\left|\int_{\mathbb{T} } g_{s}\partial _{y}g_{s}dx \big{|} _{y=0}  \right|\nonumber \\
 &&\leq \iint \left| ag_{s}^{2} \right|dxdy+\iint \left| \partial _{y}ag_{s}^{2} \right|dxdy+ 2\iint \left| ag_{s}\cdot \partial _{y} g_{s}  \right|dxdy
  \nonumber \\
   &&\quad+\iint (\partial _{y}g_{s}+g_{s})\partial_{x}^{s}(\partial _{x}p-U)dxdy+\iint g_s(\frac{\partial^{2} _{y}w}{w}+\frac{\partial^{3} _{y}w}{w}-\frac{\partial^{2} _{y}w\partial _{y}w}{w^2})\partial_{x}^{s}Udxdy
  \nonumber \\
  &&\quad+\iint \partial_y g_s\frac{\partial^{2} _{y}w}{w}\partial_{x}^{s}Udxdy
  \nonumber \\
&& \leq  C_{ \gamma,\sigma } \|w\|_{H^{s,\gamma}_{g}}^2(1+\|\partial_{x}^{s}U\|_{{L^{\infty}}(\mathbb{T})}^2)
+C\|\partial_{x}^{s}(\partial _{x}p-U)\|_{{L^{2}}(\mathbb{T})}^2
+\frac{1}{4}\|(1+y)^{\gamma}\partial_{y} g_s\|_{L^{2}}^{2} ,  \label{3.11}
\end{eqnarray}
which, together with $(\ref{2.60})$, gives the facts that $\|a\|_{L^{\infty}}\leq \delta^{-2}$ and $\|\partial _{y}a\|_{L^{\infty}}\leq \delta^{-2}+\delta^{-4}$.
Next, using H$\ddot{o}$lder inequality,  we obtain
\begin{align}
\left| 2\gamma \iint(1+y)^{2\gamma-1}g_{s}\partial _{y}  g_{s}  \right|
 &\leq 2\gamma \|\frac{1}{1+y }  \|_{L^{\infty}} \|(1+y)^{\gamma} g_{s}\|_{L^{2}} \|(1+y)^{\gamma}\partial _{y}g_{s}\|_{L^{2}} \nonumber \\
& \leq  C_{ \gamma }  \|w\|_{H^{s,\gamma}_g}^2+\frac{1}{4}\|(1+y)^{\gamma}\partial_{y} g_{s}\|_{L^{2}}^{2} ,  \label{3.12}
\end{align}
and using Lemma \ref{y4.4}
\begin{align}
\left|   \gamma \iint(1+y)^{2\gamma-1} v  |g_{s}|^{2} \right|
& \leq \|\frac{v}{1+y}\|_{L^{\infty}}\|(1+y)^{ \gamma}g_{s}\|_{L^{2}}^{2}\nonumber\\
&\leq  C_{s, \gamma,\sigma,\delta  }(\|w\|_{H^{s,\gamma}_{g}}+\|\partial_{x}^{s}U\|_{{L^2}(\mathbb{T})}) \|w\|_{H^{s,\gamma}_g}^{2}.  \label{3.13}
\end{align}
It follows from to $(\ref{2.60})$ and $(\ref{3.5})$ that
\begin{eqnarray}
\left| 2\iint(1+y)^{2\gamma}|g_{s}|^{2}\partial _{y}a  \right|
  \leq  2  \|\partial _{y}a  \|_{L^{\infty}} \|(1+y)^{\gamma} g_{s}\|_{L^{2}} ^{2}\leq C_{ \delta }  \|w\|_{H^{s,\gamma}_g}^{2},
\label{3.14}
\end{eqnarray}
and
\begin{align}
\left|    \sum\limits_{j=1}^{s-1}\binom{s}{j} \iint(1+y)^{2\gamma}g_{s} g_{j+1}\partial _{x}^{s-j}u  \right|
& \leq   \sum\limits_{j=1}^{s-1}\binom{s}{j}     \|\partial _{x}^{s-j}u  \|_{L^{\infty}} \|(1+y)^{\gamma} g_{s}\|_{L^{2}}
   \|(1+y)^{\gamma} g_{j+1}\|_{L^{2}}\nonumber \\
& \leq   C_{s, \gamma,\sigma,\delta  }(\|w\|_{H^{s,\gamma}_{g}}+\|\partial_{x}^{s}U\|_{{L^2}(\mathbb{T})})^2\|w\|_{H^{s,\gamma}_{g}}.
\label{3.15}
\end{align}
Lastly, for $j=2,3,\cdots, s-1$,
\begin{eqnarray}
&& \left|   \sum\limits_{j=1}^{s-1}\binom{s}{j}\iint(1+y)^{2\gamma}g_{s}\partial _{x}^{s-j}v \{\partial _{y}\partial _{x}^{j }w-a \partial _{x}^{j }w \}   \right|
 \nonumber \\
&& \leq    \sum\limits_{j=1}^{s-1}\binom{s}{j}  \|(1+y)^{\gamma} g_{s}\|_{L^{2}} \|\frac{\partial _{x}^{s-j}v } {1+y}\|_{L^{\infty}}
 \left(  \|(1+y)^{\gamma+1} \partial _{y}\partial _{x}^{j }w\|_{L^{2}}+ \|(1+y)a  \|_{L^{\infty}} \|(1+y)^{\gamma} \partial _{x}^{j }w\|_{L^{2}} \right)  \nonumber \\
&& \leq C_{s, \gamma,\sigma,\delta  }(\|w\|_{H^{s,\gamma}_{g}}+\|\partial_{x}^{s}U\|_{{L^2}(\mathbb{T})})\|w\|_{H^{s,\gamma}_{g}}^2,
\label{3.16}
\end{eqnarray}
which, along with $(\ref{2.60})$, gives the fact that $\|(1+y) a \|_{L^{\infty}} \leq \delta^{-2}$.  And for $j=1$,
\begin{eqnarray}
&& \left|  \iint(1+y)^{2\gamma}g_{s}\partial _{x}^{s-1}v \{\partial _{y}\partial _{x} w-a \partial _{x} w \}   \right|
 \nonumber \\
&& \leq    \|\frac{\partial _{x}^{s-1}v +y\partial_x^s U} {1+y}\|_{L^{2}} \left(  \|(1+y)^{\gamma+1} \partial _{y}\partial _{x}w\|_{L^{\infty}}+ \|(1+y)a  \|_{L^{\infty}} \|(1+y)^{\gamma} \partial _{x} w\|_{L^{\infty}} \right)\|(1+y)^{\gamma} g_{s}\|_{L^{2}}
   \nonumber \\
&& \quad+   \|\partial_x^s U\|_{L^{\infty}(\mathbb T)}\left(  \|(1+y)^{\gamma+1} \partial _{y}\partial _{x}w\|_{L^{2}}+ \|(1+y)a  \|_{L^{\infty}} \|(1+y)^{\gamma} \partial _{x} w\|_{L^{2}} \right)\|(1+y)^{\gamma} g_{s}\|_{L^{2}}\nonumber \\
&& \leq C_{s, \gamma,\sigma,\delta  }\|w\|_{H^{s,\gamma}_{g}}^2(\|w\|_{H^{s,\gamma}_{g}}+\|\partial_{x}^{s}U\|_{{L^\infty}(\mathbb{T})}).
\label{3.17}
\end{eqnarray}
For the term $| \iint(1+y)^{2\gamma}g_{s}g_1\partial _{x}^{s}U| $, we conclude
\begin{eqnarray}
&&\left| \iint(1+y)^{2\gamma}g_{s}g_1\partial _{x}^{s}U   \right|\nonumber \\
&&\leq \|(1+y)^{\gamma} g_{s}\|_{L^{2}}\|(1+y)^{\gamma} g_{1}\|_{L^{2}}\|\partial_{x}^{s}U\|_{{L^\infty}(\mathbb{T})}\nonumber \\
&& \leq C_{s, \gamma,\sigma,\delta  }\|\partial_{x}^{s}U\|_{{L^\infty}(\mathbb{T})}(\|w\|_{H^{s,\gamma}_{g}}+\|\partial_{x}^{s}U\|_{{L^2}(\mathbb{T})})\|w\|_{H^{s,\gamma}_{g}} .
\end{eqnarray}
For last term, we have
\begin{eqnarray}
&&\left|\iint(1+y)^{2\gamma}g_{s}a \partial _{x}^{j}(u-U)\partial _{x}^{s-j+1}U  \right|\nonumber \\
&&\leq \|(1+y)^{\gamma} g_{s}\|_{L^{2}}\|(1+\gamma)a\|_{L^{\infty}}\|(1+y)^{\gamma-1} \partial _{x}^{j}(u-U)\|_{L^{2}}\|\partial_{x}^{s-j+1}U\|_{{L^\infty}(\mathbb{T})}\nonumber \\
&& \leq C_{s, \gamma,\sigma,\delta  }\|\partial_{x}^{s+1}U\|_{{L^\infty}(\mathbb{T})}(\|w\|_{H^{s,\gamma}_{g}}+\|\partial_{x}^{s}U\|_{{L^2}(\mathbb{T})})\|w\|_{H^{s,\gamma}_{g}}.
\end{eqnarray}
Hence combining $(\ref{3.11})$-$(\ref{3.17})$ and $(\ref{3.7})$ leads to (\ref{3.00003}).
\end{proof}
\subsection{Weighted $H^s$ estimate on $w$ }
Now, we can derive the weighted $H^s$ estimate on $w$ by employing Proposition \ref{p2.1} and Proposition \ref{p2.2}.
\begin{Proposition}\label{p2.3}
Under the same assumption of Proposition \ref{p2.1}, we have the following estimate
\begin{align}
\|w\|_{H^{s,\gamma}_{g}}^2 &\leq \left\{\|w_0\|_{H^{s,\gamma}_{g}}^2+\int_0^t F(\tau)d \tau\right\}\times \left\{1-C(\frac{s}{2}-1)\left(\|w_0\|_{H^{s,\gamma}_{g}}^2+\int_0^t F(\tau) d\tau\right)^{\frac{s-2}{2}}t\right\}^{-\frac{2}{s-2}},
\label{2.011}
\end{align}
provided that
\begin{eqnarray*}
1-C(\frac{s}{2}-1)\left(\|w_0\|_{H^{s,\gamma}_{g}}^2+\int_0^t F(\tau) d\tau\right)^{\frac{s-2}{2}},\;\; t > 0,
\end{eqnarray*}
where $C$ is a constant independent of $\epsilon$ and $t$. The function $F(t)$ is expressed by
\begin{eqnarray}
\begin{aligned}
F(t)=\mathcal{P}(\|\partial_{x}^{s+1}U\|_{{L^\infty}(\mathbb{T})})+C_s \sum\limits_{l=0}^{\frac{s}{2}}\|\partial^{l}_t(\partial_xP-U)\|_{H^{s-2l}(\mathbb{T})}^{2},
\end{aligned}
\end{eqnarray}
and $\mathcal{P}(\cdot)$ denotes a polynomial.
\begin{proof}
According to Proposition \ref{p2.1} and Proposition \ref{p2.2}, we infer from the definition of $\|\cdot\|_{H^{s,\gamma}_g}$ that
\begin{align}
\frac{d}{dt} \|w\|_{H^{s,\gamma}_{g}}^2 &\leq C_{s, \gamma,\sigma,\delta  }\left(1+\|w\|_{H^{s,\gamma}_{g}}+\|\partial_{x}^{s}U\|_{{L^\infty}(\mathbb{T})}\right)
\left(\|w\|_{H^{s,\gamma}_{g}}+\|\partial_{x}^{s+1}U\|_{{L^\infty}(\mathbb{T})}\right)\|w\|_{H^{s,\gamma}_{g}}\nonumber\\
&\quad+C_{ \gamma,\delta  }\|\partial_{x}^{s}U\|_{{L^{\infty}}(\mathbb{T})}^2 \|w\|_{H^{s,\gamma}_{g}}^2+C_{s,\gamma,\sigma,\delta}(1+\|w\|_{H_g^{s,\gamma}})^{s-2}\|w\|^2_{H^{s,\gamma}_g}\nonumber\\
&\quad+C_s \sum\limits_{l=0}^{\frac{s}{2}}\|\partial^{l}_t(\partial_xP-U)\|_{H^{s-2l}(\mathbb{T})}^{2}\nonumber\\
&\leq C_{s,\gamma,\sigma,\delta}\|w\|_{H_g^{s,\gamma}}^{s}+\mathcal{P}(\|\partial_{x}^{s+1}U\|_{{L^\infty}(\mathbb{T})})+C_s \sum\limits_{l=0}^{\frac{s}{2}}\|\partial^{l}_t(\partial_xP-U)\|_{H^{s-2l}(\mathbb{T})}^{2},
\end{align}
and hence, it follows by using the comparison principle of ordinary differential equations that
\begin{eqnarray}
\begin{aligned}
\|w\|_{H^{s,\gamma}_{g}}^2 &\leq \left\{\|w_0\|_{H^{s,\gamma}_{g}}^2+\int_0^t F(\tau)d \tau\right\}\\
&\quad \times \left\{1-C(\frac{s}{2}-1)\left(\|w_0\|_{H^{s,\gamma}_{g}}^2+\int_0^t F(\tau) d\tau\right)^{\frac{s-2}{2}}t\right\}^{-\frac{2}{s-2}}.
\end{aligned}
\end{eqnarray}
This  hence proves Proposition \ref{p2.3}.
\end{proof}
\end{Proposition}
\subsection{Weighted $L^{\infty}$ estimates on lower order terms  }\label{su.4}
In this subsection, we will estimate the weighted $L^{\infty}$ on $D^{\alpha} w$ for $|\alpha|\leq 2$ by using the classical maximum principles. More precisely, we will derive two parts: an $L^\infty$-estimate on  $I:=\sum \limits_{|\alpha|\leq 2} |(1+y)^{\gamma+\alpha_{2}}D^{\alpha}w|_{L^{2}}$ and a lower bound estimate on $B_{(0,0)}:=(1+y)^{\sigma} w$.
\begin{Lemma}\label{y2.3}
Under the same assumption of Proposition \ref{p2.1}, we have the following estimate: \\
For any $s \geq 4$,
\begin{eqnarray}
\begin{aligned}
\|I(t)\|_{L^\infty(\mathbb{T}\times\mathbb{R}^+)}\leq \max \left\{           \|I(0)\|_{L^\infty(\mathbb{T}\times\mathbb{R}^+)},6C^2W(t)^2
\right\} e^{C\left(1+G(t)\right)t},
\label{3.31}
\end{aligned}
\end{eqnarray}
and for any $s\geq 6$,
\begin{eqnarray}
\begin{aligned}
\|I(t)\|_{L^\infty(\mathbb{T}\times\mathbb{R}^+)}\leq  \left\{           \|I(0)\|_{L^\infty(\mathbb{T}\times\mathbb{R}^+)}+C(1+W(t))W(t)^2t
\right\} e^{C\left(1+G(t)\right)t}.
\label{3.32}
\end{aligned}
\end{eqnarray}
In addition, if $s \geq 4$, we also have,
\begin{eqnarray}
\begin{aligned}
\mathop{\min}\limits_{\mathbb{T}\times\mathbb{R}^+}(1+y)^{\sigma}w(t)\geq \max \left\{         1-C\left(1+G(t)\right)te^{C\left(1+G(t)\right)t}
\right\}\cdot\min \left\{\mathop{\min}\limits_{\mathbb{T}\times\mathbb{R}^+}(1+y)^{\sigma}w_0-CW(t)t
\right\},
\label{3.33}
\end{aligned}
\end{eqnarray}
where positive constant $C$ depends on $s, \gamma,\sigma$, and $\delta$ only. The functions W and $G : [0,T] \rightarrow \mathbb{R}^+$ are respectively defined by
\begin{eqnarray}
G(t):=\mathop{\sup}\limits_{ [0,t]}\|w(\tau)\|_{H^{s,\gamma}_{g}}+\mathop{\sup}\limits_{ [0,t]}\|\partial_{x}^{s}U(\tau)\|_{{L^2}(\mathbb{T})} \quad and \quad W(t):=\mathop{\sup}\limits_{ [0,t]}\|w(\tau)\|_{H^{s,\gamma}_{g}}.
\label{2.012}
\end{eqnarray}
\end{Lemma}
\begin{proof}
For simplicity, we denote,
\begin{eqnarray*}
\begin{aligned}
I:=\sum \limits_{|\alpha|\leq 2} |(1+y)^{\gamma+\alpha_{2}}D^{\alpha}w|_{L^{2}}
\end{aligned}
\end{eqnarray*}
and
\begin{eqnarray*}
\begin{aligned}
B_{\alpha}:=(1+y)^{\gamma+\alpha_{2}}D^{\alpha}w,
\end{aligned}
\end{eqnarray*}
then $B_{(0,0)}=(1+y)^{\gamma}w$.
By a direct computation, we obtain
\begin{eqnarray}
\begin{aligned}
\left(\partial_t+u\partial_x+v\partial_y-\partial^2_y-\epsilon^2\partial^2_x+1\right)B_{\alpha}=\sum\limits_{i=1}^{3} S_{i},
\label{3.34}
\end{aligned}
\end{eqnarray}
where
\begin{eqnarray*}
\begin{aligned}
S_1=\left(\frac{\sigma+\alpha_2}{1+y}v+\frac{(\sigma+\alpha_2)(\sigma+\alpha_2-1)}{(1+y)^2}\right)B_{\alpha},\ \ S_2=-\frac{2(\sigma+\alpha_2)}{1+y}\partial_y B_{\alpha},
\end{aligned}
\end{eqnarray*}
and
\begin{eqnarray*}
\begin{aligned}
S_3=-\sum\limits_{0\leq \beta < \alpha}\binom{\alpha}{\beta}\left\{(1+y)^{\beta_2}(D^{\beta}u B_{\alpha-\beta+e_1}+\frac{D^{\beta} v B_{\alpha-\beta+e_2}}{1+y})\right\}.
\end{aligned}
\end{eqnarray*}
Multiplying the equation (\ref{3.34}) by $2B_{\alpha}$, we arrive at
\begin{eqnarray}
&&\left(\partial_t+u\partial_x+v\partial_y-\epsilon^2\partial^2_x-\partial^2_y+1\right)I\nonumber\\
&&=2\sum \limits_{|\alpha|\leq 2} \left(|S_1 B_{\alpha}|+|S_2 B_{\alpha}|+|S_3 B_{\alpha}|-\epsilon^2|\partial_x B_{\alpha}|^2-|\partial_y B_{\alpha}|^2\right)\nonumber\\
&&\leq  C_{s, \gamma,\sigma,\delta  }( 1+\|w(s)\|_{H^{s,\gamma}_{g}}+\|\partial_{x}^{s}U\|_{{L^2}(\mathbb{T})})I+C_{\delta}I+\sum \limits_{|\alpha|\leq 2}|\partial_y B_{\alpha}|^2-2\epsilon^2\sum \limits_{|\alpha|\leq 2}|\partial_x B_{\alpha}|^2\nonumber\\
&&\quad -2\sum \limits_{|\alpha|\leq 2}|\partial_y B_{\alpha}|^2+2\|(1+y)^{\beta_2}(D^{\beta}u +\frac{D^{\beta} v }{1+y})\|_{L^{\infty}}\sum \limits_{|\alpha|\leq 2}\left(\sum\limits_{0\leq \beta < \alpha}|B_{\alpha-\beta+e_1}+B_{\alpha-\beta+e_2}|\right)B_{\alpha}\nonumber\\
&&\leq  C_{s, \gamma,\sigma,\delta  }( 1+\|w(s)\|_{H^{s,\gamma}_{g}}+\|\partial_{x}^{s}U\|_{{L^2}(\mathbb{T})})I,
\end{eqnarray}
where we have used Lemma \ref{y4.4} and Young inequality.
Applying the classical maximum principle for parabolic equations, we have
\begin{align}
&\|I(t)\|_{L^\infty(T\times\mathbb{R}^+)}\nonumber\\
&\leq \max \left\{            e^{C\left(1+G(t)\right)t}\|I(0)\|_{L^\infty(T\times\mathbb{R}^+)},\mathop{\max}\limits_{\tau \in [0,t]}\left(e^{C\left(1+G(t)\right)(t-\tau)}\|I(\tau)|_{y=0}\|_{L^\infty(\mathbb{T})}\right)
\right\}.
\label{3.36}
\end{align}
 To derive an estimate  of lower bound on $B_{(0,0)}$, we arrive at
 \begin{eqnarray}
\begin{aligned}
\left(\partial_t+u\partial_x+(v+\frac{2\sigma}{1+y})\partial_y-\partial^2_y-\epsilon^2\partial^2_x+1\right)B_{(0,0)}=\left(\frac{\sigma}{1+y}v+\frac{\sigma(\sigma-1)}{(1+y)^2}\right)B_{(0,0)}.
\label{3.37}
\end{aligned}
\end{eqnarray}
Applying the classical maximum principle for (\ref{3.37}), we get
\begin{align}
&\mathop{\min}\limits_{\mathbb{T}\times\mathbb{R}^+}(1+y)^{\sigma}w(t)\nonumber\\
&\geq \max \left\{         1-C\left(1+G(t)\right)te^{C\left(1+G(t)\right)t}
\right\}\times\min \left\{\mathop{\min}\limits_{\mathbb{T}\times\mathbb{R}^+}(1+y)^{\sigma}w_0,        \mathop{\min}\limits_{[0,t]\times\mathbb{T}}w|_{y=0}
\right\}.
\label{3.38}
\end{align}
Next we start estimating the boundary values. Using  Lemma \ref{y4.3}, we obtain
\begin{align}
\|I(\tau)|_{y=0}\|_{L^\infty(\mathbb{T})}&\leq 3C^2\sum \limits_{|\alpha|\leq 2}\left(\|D^{\alpha}w\|_{L^2}^2+\|\partial_xD^{\alpha}w\|_{L^2}^2+\|\partial_y^2 D^{\alpha}w\|_{L^2}^2\right)\nonumber\\
&\leq 6C^2\|w\|_{H^{s,\gamma}_{g}}^{2},
\end{align}
which, together with (\ref{3.36}), gives (\ref{3.31}).

According to (\ref{1.5}) and boundary condition $\partial^{\alpha_1}_x u=\partial^{\alpha_1}_x v=0$, we have
\begin{eqnarray}
\begin{aligned}
\partial_t D^{\alpha}w|_{y=0}=D^{\alpha}(\epsilon^2\partial_x^2+\partial_y^2 w-w-u\partial_x w-v\partial_y w)\big{|}_{y=0}.
\end{aligned}
\end{eqnarray}
When $\alpha_2=0$ and $\alpha_1 \leq 2$,
\begin{eqnarray}
\begin{aligned}
\partial_t D^{\alpha}w\big{|}_{y=0}=(\epsilon^2\partial^{\alpha_1+2}_{x} w+\partial^{\alpha_1}_{x}\partial_y^2 w-\partial^{\alpha_1}_{x}w)\big{|}_{y=0}.
\end{aligned}
\end{eqnarray}
When $1 \leq \alpha_2 \leq 2$ and $\alpha_1=0$,
\begin{align}
\partial_t D^{\alpha}w\big{|}_{y=0}
&=\left(\epsilon^2\partial^{\alpha_2}_{y}\partial^{2}_{x} w+\partial^{\alpha_2} _{y}\partial _{y}^{2}w-\partial^{\alpha_2} _{y}w-\partial^{\alpha_2} _{y}(u\partial_x w+v\partial_y w)\right)\big{|}_{y=0}\nonumber\\
&=\left(\epsilon^2\partial^{\alpha_2}_{y}\partial^{2}_{x} w+\partial^{\alpha_2} _{y}\partial _{y}^{2}w-\partial^{\alpha_2} _{y}w-\alpha_2 w \partial_x \partial^{\alpha_2-1} _{y}w\right)\big{|}_{y=0}.
\end{align}
When $\alpha_2=1$ and $\alpha_1 =1$,
\begin{eqnarray}
\begin{aligned}
\partial_t D^{\alpha}w\big{|}_{y=0}=(\epsilon^2\partial_{y}\partial^{3}_{x} w+\partial _{x}\partial _{y}^{3}w-\partial _{x}\partial _{y}w-\partial _{x}w\partial _{x}w-w\partial _{x}^2w)|_{y=0}.
\end{aligned}
\end{eqnarray}
For $s\geq 6$, using Lemma \ref{y4.4}, we get
\begin{align}
\|\partial_t I|_{y=0}\|_{{L^\infty}(\mathbb{T})}&\leq C_{s, \gamma  } \|D^{\alpha}w\partial_t D^{\alpha}w|_{y=0}\|_{{L^\infty}(\mathbb{T})}\nonumber\\
&\leq  C_{s, \gamma  }\|w(\tau)\|_{H^{s,\gamma}_{g}}(\|w(\tau)\|_{H^{s,\gamma}_{g}}+\|w(\tau)\|_{H^{s,\gamma}_{g}}^2).
\end{align}
Hence, a direct integration yields
 \begin{eqnarray}
\begin{aligned}
\|I(t)|_{y=0}\|_{{L^\infty}(\mathbb{T})}\leq  \|I(0)|_{y=0}\|_{{L^\infty}(\mathbb{T})}+C_{s, \gamma  } (1+W(t))W(t)^2t,
\end{aligned}
\end{eqnarray}
which, together with (\ref{3.36}), gives (\ref{3.32}). For $s\geq4$,
 \begin{eqnarray}
\begin{aligned}
\|\partial_t w|_{y=0}\|_{{L^\infty}(\mathbb{T})}&=\|(\partial_y^2 w-w)|_{y=0}\|_{{L^\infty}(\mathbb{T})}\\
&\leq  CW(t).
\end{aligned}
\end{eqnarray}
Thus, a direct estimate implies
 \begin{eqnarray}
\begin{aligned}
\mathop{\min}\limits_{\mathbb{T}}{w(t)}|_{y=0}\geq \mathop{\min}\limits_{\mathbb{T}}{w_0}|_{y=0}-CW(t)t,
\end{aligned}
\end{eqnarray}
which, along with (\ref{3.38}), gives (\ref{3.33}). The proof of Lemma \ref{y2.3} is thus complete.
\end{proof}

\section{Local-in-time existence and uniqueness of solutions}
\subsection{Local-in-time existence of solutions }
In this subsection, we go back to use the symbol $(u^\epsilon, v^\epsilon, w^\epsilon)$ instead of $(u, v, w)$ from Subsections \ref{su.1}-
\ref{su.4} to denote the solution to the regularized system (\ref{2.0001}). To obtain the local-in-time solution of the initial-boundary value problem (\ref{1.1})-(\ref{1.2}), we will construct the solution to the Prandtl equations (\ref{1.1}) by passing to the limit $\epsilon \rightarrow 0^+$ in the regularized Prandtl equations (\ref{2.0001}). We only sketch the proof into five steps and more details can be found in \cite{[2]}.\\
$\textbf{Step 1.}$ According to the definition of $F$, assumption (\ref{1.8}), and the regularized Bernoulli's law,
\begin{eqnarray}
\|F\|_{L^\infty} \leq M < +\infty.
\end{eqnarray}
Thus we derive the uniform estimate for any $\epsilon \in [0,1]$ and any $t \in [0,T_1]$
\begin{eqnarray}
\|w^\epsilon\|_{H^{s,\gamma}_{g}} \leq 4 \|w_0^\epsilon\|_{H^{s,\gamma}_{g}},
\label{5.5}
\end{eqnarray}
provided that $T_1$ is chosen by (\ref{2.011}) such that
\begin{eqnarray*}
T_1 :=\min \left\{\frac{3\|w_0\|_{H^{s,\gamma}_{g}}^2}{C_{s, \gamma,\sigma,\delta  }M},\frac{1-2^{2-s}}{2^{s-2}C_{s, \gamma,\sigma,\delta  }\|w_0\|_{H^{s,\gamma}_{g}}^{s-2}}\right\}.
\end{eqnarray*}
$\textbf{Step 2.}$ When $s \geq 6$, we know from definition (\ref{2.012}) of $\Omega$ and $G$ that for any $t \in [0,T_1]$,
\begin{eqnarray}
\Omega(t) \leq 4 \|w_0^\epsilon\|_{H^{s,\gamma}_{g}} \quad and \quad G(t) \leq 4 \|w_0^\epsilon\|_{H^{s,\gamma}_{g}}+M.
\label{2.015}
\end{eqnarray}
Thus, if we choose
\begin{eqnarray*}
T_2 :=\min \left\{T_1,\frac{1}{64\delta^2 C_{s, \gamma }(1+4 \|w_0\|_{H^{s,\gamma}_{g}}) \|w_0\|_{H^{s,\gamma}_{g}}^2},\frac{\ln 2}{C_{s, \gamma,\sigma,\delta  }(1+4 \|w_0\|_{H^{s,\gamma}_{g}}+M)}\right\},
\end{eqnarray*}
then using inequality (\ref{3.32}) and the assumption on initial data,
\begin{eqnarray*}
\sum \limits_{|\alpha|\leq 2}  |(1+y)^{\sigma+\alpha_{2}}D^{\alpha}w | ^{2}\leq\frac{1}{4\delta^{2}},
\end{eqnarray*}
we have the upper bound
\begin{eqnarray}
\|\sum \limits_{|\alpha|\leq 2}  |(1+y)^{\sigma+\alpha_{2}}D^{\alpha}w^\epsilon (t) | ^{2}\|_{L^{\infty}(\mathbb{T}\times \mathbb{R}^+)}\leq\frac{1}{\delta^{2}}
\label{2.019}
\end{eqnarray}
for all $t \in [0,T_2]$. When $s \geq 4$, from the hypothesis of the initial data $\|\omega_0\| \leq C\delta^{-1}$, we have the same estimate (\ref{2.019}) for all $t\in [0,T_2]$.\\
$\textbf{Step 3.}$
Let us choose
\begin{eqnarray*}
T_3 :=\min \left\{T_1,\frac{\delta}{8 C_{s, \gamma }\|w_0\|_{H^{s,\gamma}_{g}}},\frac{1}{6C_{s, \gamma,\sigma,\delta  }N},\frac{\ln 2}{C_{s, \gamma,\sigma,\delta  }N}\right\},
\end{eqnarray*}
where $N :=1+4\|w_0\|_{H^{s,\gamma}_{g}}+M$. Then using (\ref{2.012}) and (\ref{2.015}), we derive the uniform estimate for any $\epsilon \in [0,1]$ and any $t \in [0,T_3]$,
\begin{eqnarray}
\min \limits_{\mathbb{T}\times \mathbb{R}^+}(1+y)^{\sigma} w^\epsilon(t) \geq \delta.
\label{5.6}
\end{eqnarray}
$\textbf{Step 4.}$
In summary, the above uniform estimates (\ref{5.5}),  (\ref{2.019})-(\ref{5.6}) hold for any $t \in [0,T]$, if $T$ is chosen to satisfy $T=\min\{T_1,T_2,T_3\}$. Further using almost equivalence relation  (\ref{4.1}), we have
\begin{eqnarray}
\sup  \limits_{0 \leq t \leq T}( \|w^\epsilon\|_{H^{s,\gamma}}+\|u^\epsilon-U\|_{H^{s,\gamma-1}}) \leq C\left( 4\|w_0\|_{H^{s,\gamma}_{g}}+\sup \limits_{0 \leq t \leq T}\|\partial_x^s U\|_{L^2}\right) < + \infty.
\label{2.016}
\end{eqnarray}
From the equation (\ref{2.0001}), system (\ref{2.0002}), (\ref{2.016}) and Lemma \ref{y4.3}, we also have $\partial_t w^{\epsilon}$	 and $\partial_t (u^{\epsilon}-U)$ are uniformly
bounded in $L^{\infty}([0,T];H^{s-2,\gamma})$ and $L^{\infty}([0,T];H^{s-2,\gamma-1})$ respectively. By Lions-Aubin lemma and the compact embedding of $H^{s,\gamma}$ in $H^{s^{\prime}}_{loc}$
, we conclude after taking a subsequence, as $\epsilon_k \rightarrow 0^+$,
\begin{equation}\left\{
\begin{array}{ll}
\omega^{\epsilon_k} \stackrel{*}{\rightharpoonup} \omega, &{\rm in} \quad L^{\infty}([0,T];H^{s,\gamma}),\\
\omega^{\epsilon_k} \rightarrow \omega, &{\rm in} \quad C([0,T];H^{s^{\prime}}_{loc}),\\
u^{\epsilon_k}-U \stackrel{*}{\rightharpoonup} u-U,  &{\rm in} \quad L^{\infty}([0,T];H^{s,\gamma-1}),\\
u^{\epsilon_k} \rightarrow u, &{\rm in} \quad C([0,T];H^{s^{\prime}}_{loc}),
\end{array}
         \right.\end{equation}
for all $s^{\prime} < s$, where
\begin{equation}\left\{
\begin{array}{ll}
w=\partial_y u \in L^{\infty} ([0,T];H^{s,\gamma})\cap \bigcap_{s^\prime <s}C([0,T];H^{s^{\prime}}_{loc}),\\
 u-U \in L^{\infty} ([0,T];H^{s,\gamma-1})\cap \bigcap_{s^\prime <s}C([0,T];H^{s^{\prime}}_{loc}).
\end{array}
         \right.\end{equation}
$\textbf{Step 5.}$
Using the local uniform
convergence of $\partial_x u^{\epsilon_k}$, we also have the pointwise convergence of $v^{\epsilon_k}:$ as $\epsilon \rightarrow 0^+$,
\begin{eqnarray}
v^{\epsilon_k}=-\int^y_0 \partial_x u ^{\epsilon_k}~ dy \rightarrow -\int^y_0 \partial_x u ~dy=: v.
\end{eqnarray}
Thus, passing the limit $\epsilon_k \rightarrow 0^+$ in the initial-boundary value problem (\ref{2.0001}), we get that the limit
$(u, v)$ solves the initial-boundary value problem (\ref{1.4}), (\ref{1.2}) in the classical sense. Furthermore, we obtain
$(u, v, b)$ solves the initial-boundary value problem (\ref{1.1})-(\ref{1.2}) in the classical sense. This completes the proof
of the existence.
\subsection{Uniqueness of Solutions}
In this subsection, we are going to prove the uniqueness of solutions to 2D magnetic Prandtl model.  Let us denote $(\bar u -\bar v)=(u_1,v_1)-(u_2,v_2)$, $\bar w =w_1-w_2$, $\bar b_1 =b_{11}-b_{12}$ and $a_2=\frac{\partial_y w_2}{w_2}$. It is easy to check that $\bar g =\bar w -a_2 \bar u=w_2 \partial_y(\frac{\bar u}{w_2})$ and the evolution equation on $\bar g$ is as follows
\begin{align}
(\partial_t+u_1\partial_x+v_1\partial_y-\partial^2_y+1)\bar g
&=(\partial_t+u_1\partial_x+v_1\partial_y-\partial^2_y+1)\bar w-a_2(\partial_t+u_1\partial_x+v_1\partial_y-\partial^2_y+1)\bar u\nonumber\\
&\quad -\bar u(\partial_t+u_1\partial_x+v_1\partial_y-\partial^2_y+1)a_2-2\partial_y a_2 \bar w.
 \label{33.40}
\end{align}
Next, we calculate the values of the first three terms on the right-hand side of equation (\ref{33.40}) respectively. Recalling the vorticity system (\ref{1.5}), we have
\begin{eqnarray}
\begin{aligned}
(\partial_t+u_i\partial_x+v_i\partial_y+1)\partial_y w_i=\partial_y^3 w_i+\partial_x u_i \partial_y w_i -w_i \partial_x w_i.
\end{aligned}
\end{eqnarray}
Then, according to the definition of $a_i$, we get
\begin{align}
(\partial_t+u_i\partial_x+v_i\partial_y+1)a_i&=\frac{(\partial_t+u_i\partial_x+v_i\partial_y+1)\partial_y w_i}{w_i}-\frac{\partial_y w_i(\partial_t+u_i\partial_x+v_i\partial_y+1) w_i}{w_i^2}\nonumber\\
&=\frac{\partial^3_y w_i}{w_i}+a_i\partial_x u_i-\partial_x w_i -a_i \frac{\partial_y^2 a_i}{w_i},
\end{align}
which, combined with the fact
\begin{eqnarray*}
\frac{\partial^3_y w_i}{w_i}-a_i \frac{\partial_y^2 a_i}{w_i}=\partial_y^2 a_i+2a_i \partial_y a_i,
\end{eqnarray*}
imolies
\begin{eqnarray}
\begin{aligned}
(\partial_t+u_i\partial_x+v_i\partial_y-\partial_y^2+1)a_i=a_i\partial_x u_i-\partial_x w_i +2a_i \partial_y a_i.
\end{aligned}
\end{eqnarray}
Furthermore, we conclude that
\begin{eqnarray}
\begin{aligned}
(\partial_t+u_1\partial_x+v_1\partial_y-\partial_y^2+1)a_2
=a_2\partial_x u_2-\partial_x w_2 +2a_2 \partial_y a_2+(\bar u \partial_x+\bar v \partial_y)a_2.
 \label{33.41}
\end{aligned}
\end{eqnarray}
For the estimate of $(\partial_t+u_1\partial_x+v_1\partial_y-\partial^2_y+1)\bar u$ and $(\partial_t+u_1\partial_x+v_1\partial_y-\partial^2_y+1)\bar w$, from the equation (\ref{1.4}) and (\ref{1.5}), we can derive that
\begin{eqnarray}
\begin{aligned}
(\partial_t+u_1\partial_x+v_1\partial_y-\partial^2_y+1)\bar u=-\bar u \partial_x u_2-\bar v \partial_y  u_2,
\end{aligned}
\end{eqnarray}
and
\begin{eqnarray}
\begin{aligned}
(\partial_t+u_1\partial_x+v_1\partial_y-\partial^2_y+1)\bar w=-\bar u \partial_x w_2-\bar v \partial_y  w_2.
 \label{33.42}
\end{aligned}
\end{eqnarray}
Using (\ref{33.40}) and (\ref{33.41})-(\ref{33.42}), we have
\begin{eqnarray}
\begin{aligned}
(\partial_t+u_1\partial_x+v_1\partial_y-\partial^2_y+1)\bar g=-2\bar w \partial_y a_2 -\bar u(\bar u \partial_x a_2+\bar v \partial_y a_2+2a \partial_y a_2).
\end{aligned}
\end{eqnarray}
Now we derive $L^2$ estimate on $\bar g$. For any $t \in (0,T]$, multiplying by $2 \bar g$ and then integrating by parts over $\mathbb{T}\times\mathbb{R}_{+}$, we obtain
\begin{align}
&\frac{d}{dt}\|\bar g\|_{L^2}^2+2\|\bar g\|_{L^2}^2+2\|\partial_y \bar g\|_{L^2}^2\nonumber\\
&=\int_{\mathbb{T}}\bar g \partial_y \bar g |_{y=0}dx-2\iint \bar g \bar u(\bar u \partial_x a_2+\bar v \partial_y a_2+2a \partial_y a_2)-4\iint \bar g \bar w \partial_y a_2 \nonumber \\
&\quad -2\iint \bar g(u_1\partial_x \bar g+v_1\partial_y \bar g).
 \label{33.43}
\end{align}
We need to estimate the integral equation (\ref{33.43}) term by term, applying the simple trace theorem and Young's inequality,
\begin{align}
|\int_{\mathbb{T}}\bar g \partial_y \bar g |_{y=0}dx|
&\leq|\iint a_2 |\bar g|^2  dxdy|+|\iint \partial_y a_2 |\bar g|^2  dxdy|+2|\iint  a_2 \bar g \partial_y \bar g  dxdy|\nonumber\\
&\leq \frac{1}{2}\|\partial_y \bar g\|_{L^2}^2+C_{\sigma,\delta  }\|\bar g\|_{L^2}^2,
\end{align}
where we used the fact $ \partial_y \bar g |_{y=0}=-a_2 \bar g  |_{y=0}$ ($\partial_y \bar w=0$). We claim $\|\frac{\bar u}{1+y}\|_{L^2} \leq C_{\sigma,\delta} \|\bar g\|_{L^2}$, so by Lemma \ref{y4.4},
\begin{align}
-2\iint \bar g \bar u(\bar u \partial_x a_2+\bar v \partial_y a_2+2a \partial_y a_2) &\leq  2\|(1+y)(\bar u \partial_x a_2+\bar v \partial_y a_2+2a \partial_y a_2)\|_{L^\infty}\|\frac{\bar u}{1+y}\|_{L^2}\|\bar g \|_{L^2}\nonumber\\
&\leq C_{ \gamma,\sigma,\delta  }(1+\|w_i\|_{H^{4,\gamma}_{g}}+\|\partial_{x}^{4}U\|_{{L^2}(\mathbb{T})})\|\bar g\|_{L^2}^2.
\end{align}
Below we give the fact that $\|\frac{\bar u}{1+y}\|_{L^2}$ can be controlled by $\|\bar g\|_{L^2}$, since $\delta \leq (1+y)^{\delta} w_2\leq \delta^{-1}$,
\begin{eqnarray}
\begin{aligned}
\|\frac{\bar u}{1+y}\|_{L^2} \leq \delta^{-1} \|(1+y)^{-\sigma-1}\frac{\bar u}{w^2}\|_{L^2} \leq C_{\sigma,\delta}\|(1+y)^{-\sigma} \partial_x(\frac{\bar u}{w_2})\|_{L_2} \leq C_{\sigma,\delta} \|\bar g\|_{L^2}.
\end{aligned}
\end{eqnarray}
In addition, we also have
\begin{eqnarray}
\begin{aligned}
\| \bar w\|_{L^2} \leq \|\bar g \|_{L^2}+\delta^{-2}\|\frac{\bar u}{1+y}\|_{L^2}\leq C_{\sigma,\delta} \|\bar g\|_{L^2},
\end{aligned}
\end{eqnarray}
and
\begin{eqnarray}
\begin{aligned}
-4\iint  \bar g \bar w \partial_y a_2 \leq C_{\sigma,\delta} \|\bar g\|_{L^2}^2.
\end{aligned}
\end{eqnarray}
For the last term in (\ref{33.43}), using the integration by parts, boundary condition $(u_1, v_1)\big{|}_{y=0}$ and $\partial_x u_1+\partial_y v_1=0$, we can  show
\begin{eqnarray}
\begin{aligned}
-2\iint \bar g(u_1\partial_x \bar g+v_1\partial_y \bar g)=0.
\label{5.2}
\end{aligned}
\end{eqnarray}
Combining all of the above estimates (\ref{33.43})-(\ref{5.2}), we conclude
\begin{eqnarray}
\begin{aligned}
\frac{d}{dt}\|\bar g\|_{L^2}^2 \leq C_{ \gamma,\sigma,\delta  }(1+\|w_i\|_{H^{4,\gamma}_{g}}+\|\partial_{x}^{4}U\|_{{L^2}(\mathbb{T})})\|\bar g\|_{L^2}^2.
\end{aligned}
\end{eqnarray}
 which, by Gronwall's inequality, gives
\begin{eqnarray}
\begin{aligned}
\|\bar g(t)\|_{L^2}^2 \leq \|\bar g(0)\|_{L^2}^2e^{Ct},
\end{aligned}
\end{eqnarray}
with $C=C_{ \gamma,\sigma,\delta  }(1+\|w_i\|_{H^{4,\gamma}_{g}}+\|\partial_{x}^{4}U\|_{{L^2}(\mathbb{T})})$, and this implies $\bar g=0$ due to $u_1\big{|}_{t=0}=u_2\big{|}_{t=0}$. Since $w_2\partial_y (\frac{u_1-u_2}{w_2})=\bar g =0$, we derive
\begin{eqnarray}
\begin{aligned}
u_1-u_2=q w_2
\end{aligned}
\end{eqnarray}
for some function $q=q(t,x)$. Now using the  Oleinik's monotonicity assumption $w_2 >0 $ and boundary condition $u_1\big{|}_{y=0}=u_2\big{|}_{y=0}=0$, we can get $q=0$, and hence $u_1=u_2$. Furthermore, by using $\partial_x u+\partial_y v=0$ and $\partial_y u+\partial_y^2 b_1=0$, then $v_i$  and $b_{1i}$ can be uniquely determined (i.e., $v_1=v_2$ and $b_{11}=b_{12}$).

{\bf Acknowledgement}
This paper was supported in part by the NNSF of China with contract  number 12171082, the fundamental research funds for the central universities with contract numbers 2232022G-13, 2232023G-13 and  the Scientific Research Foundation of Education Department of Yunnan province with  contract  number 2023J0133.


\newpage
\begin{thebibliography}{lllp}
\setlength{\itemsep}{- 2mm}
\bibitem{[5]} R. Alexander, Y. Wang, C. Xu and T. Yang, Well posedness of the Prandtl eqauation
in Sobolev spaces, J. Amer. Math. Soc., 28(3)(2015), 745-784.

\bibitem{[28]} R.E. Caflisch and M. Sammartino, Existence and singularities for the Prandtl boundary layer equations, Z. Angew.
Math. Mech.,  80(11-12)(2000), 733-744.

\bibitem{[36]} M. Cannone,  M.C. Lombardo and M. Sammartino,  Existence and uniqueness for the Prandtl equations, C. R. Acad. Sci. Paris, 332(3)(2001), 277-282.

\bibitem{[17]} D. Chen, Y. Wang and Z. Zhang, Well-posedness of the linearized Prandtl equation around
a non-monotonic shear flow, Ann. I. H. Poincar$\acute{e}$-AN, 35(2018), 1119-1142.

\bibitem{[18]} D. Chen, Y. Wang and Z. Zhang, Well-posedness of the Prandtl equation with monotonicity in Sobolev spaces, J. Differential Equations, 264(2018), 5870-5893.

\bibitem{[29]} W.F. Cope and D.R. Hartree, The laminar boundary layer in compressible flow, Philos. Trans. R. Soc. A., 241(827)(1951), 1-69.

\bibitem{[22]} W. E, Boundary layer theory and the zero-viscosity limit of the Navier-Stokes equation, Acta Math. Sin., 16(2000),
207-218.

\bibitem{[30]} W. E and  B.  Engquist, Blow up of solutions of the unsteady Prandtl's equation, Comm. Pure Appl. Math., 50(12)(1997), 1287-1293.

\bibitem{[21]} H. Dietert and D. G$\acute{e}$rard-Varet, Well-Posedness of the Prandtl equations without any structural assumption, Ann. Partial Differential Equations, (2019) 5:8,
https://doi.org/10.1007/s40818-019-0063-6.

\bibitem{[15]} J. Gao, D. Huang and Z. Yao, Boundary layer problems for the Two-dimensional inhomogeneous incompressible magnetohydrodynamics equations, Anal. Partial Differential Equations, (2018), arXiv:1810.11258v2.

\bibitem{[31]} F. Gargano, M. Sammartino and V. Sciacca, Singularity formation for Prandtl's equations, Phys. D.,  238(19)(2009),
1975-1991.

\bibitem{[32]} D. G$\acute{e}$rard-Varet and E. Dormy, On the ill-posedness of the Prandtl equation, J. Amer. Math. Soc., 23(2)(2010), 591-609.

\bibitem{[33]} D. G$\acute{e}$rard-Varet, Y. Maekawa and N. Masmoudi, Gevrey stability of Prandtl expansions for 2D Navier-Stokes, (2016), arXiv:
1607.06434v1.

\bibitem{[34]} D. G$\acute{e}$rard-Varet and N. Masmoudi, Well-posedness for the Prandtl system without analyticity or monotonicity, Ann.
Sci. $\acute{E}$c. Norm. Sup¨¦r., 48(6)(2015), 1273-1325.

\bibitem{[35]} D. G$\acute{e}$rard-Varet and T. Nguyen, Remarks on the ill-posedness of the Prandtl equation, Asymptot. Anal., 77(1-2)(2012),
71-88.

\bibitem{[4]} D. G$\acute{e}$rard-Varet and M. Prestipino, Formal derivation and stability analysis of boundary layer models in MHD, Z. Angew. Math. Phys., 68(3)(2017), Art., 76.


\bibitem{[7]} S. Gong, Y. Guo and Y. Wang, Boundary layer problems for the two-dimensional compressible Navier-Stokes equations,
Anal. Appl., 14(1)(2016),  1-37.

\bibitem{[23]} Y. Guo and T. Nguyen, A note on Prandtl boundary layers, Comm. Pure Appl. Math., 64(10)(2011), 1416-1438.

\bibitem{[11]} Y. Guo and T. Nguyen, Prandtl boundary layer expansions of steady Navier-Stokes flows over a moving plane, Ann. Partial Differential Equations, 3(10)(2017), DOI: 10.1007/s40818-016-0020-6.

\bibitem{[16]} Q. Hou, C. Liu, Y. Wang and Z. Wang, Stability of boundary layers for a viscous hyperbolic system arising from chemotaxis: one-dimensional case, SIAM J. Math. Anal., 50(3)(2018), 3058-3091.


\bibitem{[20]} Y. Huang, C. Liu and T. Yang,  Local-in-time well-posedness for compressible MHD boundary layer, J. Differential Equations, 266(2019), 2978-3013.

\bibitem{[6]} M. Ignatova and V. Vicol, Almost global existence for the Prandtl boundary layer equations,
Arch. Ration. Mech. Anal., 220(2)(2016), 809-848.

\bibitem{[24]} I. Kukavica, N. Masmoudi, V. Vicol and T.K. Wong, On the local well-posedness of the Prandtl and the hydrostatic
Euler equations with multiple monotonicity regions, SIAM J. Math. Anal., 46(2014), 3865-3890.

\bibitem{[8]} W. Li, D. Wu and C. Xu,  Gevrey class smoothing effect for the Prandtl equation, SIAM J. Math. Anal., 48(3)(2016), 1672-1726.

\bibitem{[12]} W. Li and  T. Yang, Well-posedness in Gevrey space for the Prandtl equations with non-degenerate critical points, J. Eur. Math. Soc.,  22(3)(2020), 717-775.

\bibitem{[13]} C. Liu, F. Xie and  T. Yang, MHD boundary layers in Sobolev spaces without monotonicity, I. Well-posedness theory, Analysis of PDEs, (2016), arXiv 1611.05815v4.

\bibitem{[14]} C. Liu, F. Xie and T. Yang, MHD boundary layers in Sobolev spaces without monotonicity, II. Convergence theory, Analysis of PDEs, (2018), arXiv:1704.00523v4.

\bibitem{[25]} M.C. Lombardo, M. Cannone and M. Sammartino, Well-posedness of the boundary layer equations, SIAM J. Math.
Anal., 35(4)(2003), 987-1004.


\bibitem{[2]} N. Masmoudi and T. K. Wong, Local-in-time existence and uniqueness of solutions to the Prandtl equations by energy methods, Comm. Pure Appl. Math.,
    68(10)(2015), 1683-1741.

\bibitem{[1]} M. Medve$\check{d}$,  A new approach to an analysis of Henry type integral inequalities and their Bihari type versions, J. Math. Anal. Appl.,
214(2)(1997), 349-366.

\bibitem{[26]} O.A. Oleinik and V.N. Samokhin, Mathematical Models in Boundary Layer Theory, Applied Mathematics and Mathematical Computation, vol. 15,
 Chapman and Hall/CRC, Boca Raton, FL, 1999.

\bibitem{[37]} Y. Qin, Integral and discrete inequalities and their applications, Vol. I. Nonlinear inequalities; Vol. II. Nonlinear inequalities,
  Springer International
Publishing AG, Birkhauser, 2016.

\bibitem{[38]} Y. Qin, Analytic inequalities and their applications in PDEs, Springer International
Publishing, 2017.
\bibitem{[39]}Y. Qin and X.Dong, Local well-posedness of solutions to 2D mixed Prandtl equations in Sobolev space without monotonicity and lower bound, submitted.
\bibitem{[3]} F. Xie and Y. Tong, Global-in-time stability of 2D MHD boundary Layer in the Prandtl-Hartmann regime, SIAM J. Math. Anal., 50(6)(2018), 5749-5760.

\bibitem{[19]} F. Xie and T. Yang, Lifespan of solutions to MHD boundary layer equations
with analytic perturbation of general shear flow, Acta Math. Appl. Sin-E., 35(1)(2019), 209-229.

\bibitem{[27]} Z. Xin and L. Zhang, On the global existence of solutions to the Prandtl's system, Adv. Math.,  181(2004), 88-133.

\bibitem{[10]} C. Xu and X. Zhang, Long time well-posedness of Prandtl equations in Sobolev space, J. Differential Equations, 263(2017), 8749-8803.
\bibitem{[9]} P. Zhang,  Z.  Zhang, Long time well-posedness of Prandtl system with small and analytic initial data, J. Funct. Anal., 270(7)(2016), 2591-2615.
\end{thebibliography}
\end{document}